\documentclass[12pt]{amsart}
\usepackage{amsfonts,amsmath,amssymb,amsthm}
\allowdisplaybreaks[1]
\usepackage{tikz-cd}
\usepackage{tikz-3dplot}
\usepackage{subcaption}
\usepackage{extarrows}
\usepackage{enumerate}
\usepackage{bm}
\usepackage[a4paper, margin=1in]{geometry}
\usepackage{hyperref} 

\hypersetup{
    colorlinks=true,
    linkcolor=black,
    filecolor=magenta,      
    urlcolor=cyan,
    pdftitle={Your Document Title},
    bookmarks=true
}

\usepackage{dsfont}
\usepackage{cleveref} 

\crefname{theorem}{theorem}{theorems}
\Crefname{theorem}{Theorem}{Theorems}
\crefname{definition}{definition}{definitions}
\Crefname{definition}{Definition}{Definitions}

\newtheorem{theorem}{Theorem}[section]
\newtheorem{lemma}[theorem]{Lemma}
\newtheorem{proposition}[theorem]{Proposition}
\newtheorem{corollary}[theorem]{Corollary}
\theoremstyle{definition}
\newtheorem{definition}[theorem]{Definition}
\newtheorem{question}[theorem]{Question}
\newtheorem{example}[theorem]{Example}

\theoremstyle{remark}
\newtheorem{remark}[theorem]{Remark}

\begin{document}
\begin{sloppypar}
\title{Asymptotic behavior of modular representations over abelian $p$-groups}
\author{Cheng Meng}
\address{Cheng Meng\\ Yau Mathematical Sciences Center, Tsinghua University, Beijing 100084, China. \emph{Email:} {\rm cheng319000@mail.tsinghua.edu.cn}}

\date{\today}
\begin{abstract}
In this paper, we prove some results on the asymptotic behavior arising in modular representation theory over abelian $p$-groups. First, we embed the representation ring of a cyclic $p$-group into a real algebra of functions. Second, we calculate the asymptotic order of the dimension of the core of $n$-th tensor power of a direct sum of syzygies and cosyzygies of the trivial module, which is of the form $C\gamma^nn^\alpha$. This result leads to a negative answer to a question by Benson and Symonds, that is, the dimension of the core of $M^{\otimes n}$ for certain $\Omega$-algebraic module $M$ is not eventually recursive. Third, we give a systematic way of computing the core series of $\Omega$-algebraic modules. Finally, we show the existence of a transcendental core series, which comes from iterated syzygy modules of the trivial representation.
\end{abstract}
\maketitle
\tableofcontents
\section{Introduction}
This paper explores some asymptotic behavior of modular representations of finite abelian groups. We focus on the following aspects:
\begin{enumerate}
\item Limits in representation rings over cyclic groups;
\item Asymptotic formula for core-dimension of syzygies and cosyzygies of the trivial representation;
\item Computation of core series of $\Omega$-algebraic modules;
\item Algebraicity and transcendence of core series.
\end{enumerate}
We will study the case where the field has characteristic $p$, and the group is a finite abelian $p$-group.
\subsection{Cyclic case} The Green ring, or representation ring over a cyclic $p$-group $\mathbb{Z}/p^e\mathbb{Z}$ has been studied in \cite{Green}, where the decomposition of tensor products of indecomposable objects is given. We call this ring $\Gamma_e$, which depends only on $e$ and the characteristic. The formula for tensor products of a general $\mathbb{Z}/p^e\mathbb{Z}$-representation can be described using multiplication in $\Gamma_e$.

From another point of view, Han and Monsky studied in \cite{HanMonsky} the representation theory of the $k$-objects, which is a finitely generated $k[T]$-module annihilated by a power of $T$. They construct the tensor product of two $k$-objects and explore the structure of the representation ring. Their theory of $k$-objects is a foundation of computation of Hilbert-Kunz multiplicity in many cases, especially the diagonal case. Works following Han and Monsky's method can be found in \cite{HM2} \cite{HM3CSTZ}\cite{arxivAnalysis}\cite{HM4SMIR}\cite{HM5}.

It can be observed that Han-Monsky's representation ring is actually equal to a direct limit of Green rings $\Gamma_e$. So, Han and Monsky's formula for tensor product and Green's formula for tensor product are equivalent.

However, when we apply Han and Monsky's method to compute the Hilbert-Kunz multiplicity, some limit behaviors often arise from a sequence $\{V_e\}_e$ where $V_e \in \Gamma_e$. For example, we may want to consider when
$$e \to \ell(e)=\ell(V_e\otimes_{k[\mathbb{Z}/p^e\mathbb{Z}]}k)$$
has the form $\ell(e)\sim Cp^{e\alpha}$, and calculate the value of $C,\alpha$. We would also like to see if certain limit behavior passes from $\{V_e\},\{W_e\}$ to the sequence of tensor products $\{V_e\otimes_kW_e\}$.
In this sense, the study of the tensor product on a single $\Gamma_e$ is not sufficient; we need to check the compatibility of the tensor product for different $\Gamma_e$'s. This relies on the communication between $\Gamma_e$ and $\Gamma_{e'}$ for $e \neq e'$, including functors like the restriction, induction and inflation.

Now we introduce the main theorem in this part which explains the meaning of ``taking limit of representations in $\Gamma_e$". Define
$$D(t_1,t_2,t_3)=\lim_{e \to \infty}\frac{\ell(k[T_1,T_2]/(T_1^{\lceil t_1p^e \rceil},T_2^{\lceil t_2p^e \rceil},(T_1+T_2)^{\lceil t_3p^e\rceil}))}{p^{2e}},$$
which exists and is a continuous function $\mathbb{R}^3 \to \mathbb{R}$, see \Cref{3.4}. Let $\mathcal{F}^+$ be the space of functions which are locally constant outside $[0,1]$, increasing and concave on $[0,\infty)$, and vanish at $0$ with bounded two-sided derivatives, and $\mathcal{F}$ be the $\mathbb{R}$-linear space generated by $\mathcal{F}^+$. For $f_1,f_2 \in \mathcal{F}$, define their product to be the real function $[0,1] \to \mathbb{R}$:
$$f(t)=f_1*f_2(t)=\int_{[0,1^+]^2}D(t_1,t_2,t)d(-f'_1(t_1))d(-f'_2(t_2)).$$
Also, define
$$\alpha_{a,e}=\max\{\min\{t,a/p^e\},0\} \in \mathcal{F}.$$
\begin{theorem}
Under the above notations,
\begin{enumerate}
\item (\Cref{3.15}) $(\mathcal{F},+,*)$ is a commutative $\mathbb{R}$-algebra.
\item (\Cref{3.14}) let $\Gamma_e$ be the representation ring of $\mathbb{Z}/p^e\mathbb{Z}$, $V_i=k[g]/(g-1)^i$ be the $i$-dimensional indecomposable representation and $\delta_i=[V_i]$ be its class in $\Gamma_e$. Then the $\mathbb{Z}$-linear map $\Gamma_e \to \mathcal{F}$ which maps $\delta_a$ to $p^e\alpha_{a,e}$ is an embedding of rings.
\item (\Cref{3.19}) there is a norm $\|\|$ on $\mathcal{F}$ making it a normed algebra.
\end{enumerate}
\end{theorem}
Here the map $\Gamma_e\to \mathcal{F}$ is compatible with the map $\frac{1}{p}\operatorname{Ind}:p\Gamma_e \to \Gamma_{e+1}$.

This main theorem explains the meaning of ``limits of $\{V_e\}$, where $V_e \in \mathbb{R}\otimes_\mathbb{Z}\Gamma_e$": we embed the class of $V_e$ into $\mathcal{F}$ and view it as a real function, and then take pointwise limit. The main theorem says we can take tensor product of these limits which is given by an integral transform.
\subsection{Asymptotic formula for core dimensions}
The idea behind the construction of $\mathcal{F}$ is to find a sequence of indecomposable objects parametrized by a countable parameter $n$ and find their asymptotic distribution with respect to $n$. For example, the indecomposable $k$-objects are exactly $k[T]/(T^n)$, which are parametrized by $n \in \mathbb{Z}_{\geq1}$.

The same idea can be run on any countable classes of representations with good asymptotic distribution. It turns out that the core of direct sums of syzygies and cosyzygies often follows a normal distribution asymptotically. Therefore, we can extract an asymptotic formula for the core dimension in terms of the mean and variance of this normal distribution.

We recall some definitions in modular representation theory. The non-projective part of a $G$-representation $M$ is called the core of $M$, denoted by $\operatorname{core}(M)$. In \cite{BensonGamma}, the authors consider the following numbers
$$c^G_n(M)=\dim_k \operatorname{core}(M^{\otimes n}),\gamma_G(M)=\lim_{n \to \infty}\sqrt[n]{c^G_n(M)}.$$
The limit always exists and we have $\log c^G_n(M)\sim n\log \gamma_G(M)$. Therefore, we may further consider the behavior of $c^G_n(M)/\gamma_G^n(M)$. The main theorem in this part says
$$c^G_n(M)/\gamma_G^n(M) \sim Cn^\alpha$$
for some constant $C,\alpha$ when $M$ is a direct sum of syzygies or cosyzygies of the trivial module. 

A precise description of $C,\gamma,\alpha$ relies on the data of $G$ and the decomposition of $M$, which we describe below. Let $r$ be the number of generators of $G$, $D=|G|$, $M=\oplus_{i \in \mathbb{Z}}\Omega^i(k)^{a_i}\neq 0$, $\gamma=\sum_i a_i$, $X$ be a random variable such that $P(X=i)=a_i/\gamma$, and the mean and the variance of $X$ are $\mu$ and $\sigma^2$ respectively. Let $Z \sim N(0,1)$ be a random variable.
\begin{theorem}[See \Cref{4.15}]
Under the above notations, if $\mu \neq 0$, then
$$c^G_n(M)\sim \gamma^n\cdot n^{r-1}\cdot |\mu|^{r-1}\cdot \frac{D}{2(r-1)!}.$$
If $\mu=0$ and $\sigma \neq 0$, then
$$c^G_n(M)\sim \gamma^n\cdot n^{(r-1)/2}\cdot \sigma^{r-1}\cdot E(|Z|^{r-1})\cdot \frac{D}{2(r-1)!}.$$
If $\mu=\sigma=0$, then
$$c^G_n(M)=\gamma^n.$$
\end{theorem}
Here the direct sum distribution of $\operatorname{core}(M^{\otimes n})$ approaches the normal distribution by the central limit theorem. Therefore, in the proof of the main theorem in this part, the language of probability theory is necessary.
A similar result on the dimension of $\operatorname{soc}\operatorname{core}(M^{\otimes n})$ is proved in \Cref{4.18}.

A consequence of the above theorem is the following theorem which answers \cite[Question 14.2]{BensonGamma} in negative. For the definition of $\Omega$-algebraic and eventually recursive sequences, see subsection 4.6 or \cite{BensonGamma}. 
\begin{theorem}[See \Cref{4.21}]
Let $G$ be an abelian $p$-group minimally generated by $r$ elements where $r$ is even. Then for $M=\oplus_{i \in \mathbb{Z}}\Omega^i(k)^{a_i}\neq 0$ with $\sum ia_i=0$, $M$ is $\Omega$-algebraic while $c^G_n(M)$ is not eventually recursive.     
\end{theorem}
\subsection{Computing core series of $\Omega$-algebraic modules}
We call
$$c_M(z)=\sum_{n \geq 0}c^G_n(M)z^n$$
the core series of $M$. When $n=0$, we adopt the convention that $M^{\otimes 0}=k$ and $c^G_n(M)=1$. \Cref{4.21} is equivalent to the fact that there is an algebraic core series which is not rational. We move on to check more cases of core series and explore their behavior. 

We first focus on the $\Omega$-algebraic case. In this case, the core series is always algebraic by a result of \cite{CHU}. In Chapter 5, we will present a systematical way of computing such series using finite data on the decomposition of tensor products. It turns out that many such series are algebraic but not rational. 

The core statement we use here is the \textbf{formal Cauchy's residue theorem} (see \Cref{5.10}), which allows us to calculate the formal contour integrals $[u^{-1}]f(u,z)$ using residues at ``small" poles of $f$. This theorem has been widely used in analytic combinatorics, and in this paper it allows us to evaluate core series very quickly. 

Our method of computation leads to the following description of the core series. Assume the core of all tensor powers of the $\Omega$-algebraic module $M$ lie in the orbit of $M_1,\ldots,M_r$, and $[\otimes M]$ gives an $r*r$-matrix $A \in \operatorname{Mat}_r(\mathbb{N}[u,u^{-1}])$ on the free $\mathbb{Z}[u,u^{-1}]$-module generated by classes $[M_i]$'s. For a polynomial in $\mathbb{C}((z))[[u]]$, a root $u=\rho(z) \in \overline{\mathbb{C}((z))}$ is called small if $\operatorname{ord}_z\rho(z)>0$, otherwise it is big. Let $N$ be a direct sum of syzygies and cosyziges of $M_i$'s.
\begin{theorem}[See \Cref{5.16}]
Let $h(u,z)=\det(1-zA)$, which can be viewed as a Laurent polynomial in $u$. Assume all small roots of $h(u,z)$ are $\rho_{1}(z),\ldots,\rho_{r}(z)$.  Then
$$\sum_{n \geq 0}\ell(N\otimes M^n)z^n \in \mathbb{C}(z,\rho_{i}(z)|1 \leq i \leq r).$$
If $h(u,z)=u^{-\nu}h_1(u,z)$ such that $h_1(u,z)$ is a polynomial in $u$ which is not divisible by $u$, $\deg_u h_1(u,z)=D$, then
$$[\mathbb{C}(z,f(z)):\mathbb{C}(z)]\leq D!.$$
If moreover $A$ is a $1*1$ matrix, then 
$$[\mathbb{C}(z,f(z)):\mathbb{C}(z)]\leq D(D-1)\ldots (D-\nu+1).$$
Also, the above holds for the core series of an $\Omega$-algebraic module.
\end{theorem}
We calculate some concrete examples of core series, including the case $(G,M)=(V_4,\Omega^2(k)\oplus (\Omega^{-1}(k))^2),(V_4,\Omega^3(k)\oplus (\Omega^{-2}(k))^2),(V_4,\Omega^{-3}(k)\oplus (\Omega^{-1}(k))^3\oplus\Omega(k)^3\oplus \Omega^3(k))$. We also calculate the core series in \cite[Example 15.1]{BensonGamma} and \cite[Section 5, example (4)]{CHU} as below. 
\begin{theorem}[See \Cref{5.22}]
Let $G=\mathbb{Z}/3\times \mathbb{Z}/3$ and $M$ is given by the following diagram  
\begin{center}
\begin{tikzpicture}[thick, line cap=round, shorten >=4pt, shorten <=4pt]

\coordinate (A) at (0,0);
\coordinate (B) at (1,1);
\coordinate (C) at (2,0);
\coordinate (D) at (-1,1);
\coordinate (E) at (-2,0);
\coordinate (F) at (-1,-1);

\draw (A) -- (B);
\draw (A) -- (D);
\draw (A) -- (F);
\draw (B) -- (C);
\draw (D) -- (E);
\draw (E) -- (F);

\fill (A) circle (2pt);
\fill (B) circle (2pt);
\fill (C) circle (2pt);
\fill (D) circle (2pt);
\fill (E) circle (2pt);
\fill (F) circle (2pt);

\node at (D) [below=6pt] {$a$};
\node at (B) [below=6pt] {$b$};

\end{tikzpicture}    
\end{center}
Then
\begin{align*}
c_M(z)=1+z\left(\frac{3(2+3z-8z^2-12z^3)}{(1-4z^2)(1-9z^2)}
+\frac{18z}{(1-2z)\sqrt{1-4z^2}}\right).
\end{align*}
\end{theorem}
\subsection{One non-$\Omega$-algebraic module and one transcendental core series}
In Section 6, we study direct sums of the following classes of modules
$$\Omega^j_S\Omega^i_R(k),\Omega^j_S(R),S$$
where $R=kH,S=kG$, $G$ surjects onto $H$, and any $H$ representation is identified with a $G$-representation via inflation; equivalently this identifies an $R$-module with an $S$-module via surjection $S \twoheadrightarrow  R$.

The above classes are closed under tensor product, but they lie in countably many orbits under $\Omega_S$, so modules in these classes are in general not $\Omega$-algebraic. We have the following result.
\begin{theorem}[See \Cref{6.4}]
Let
$$M=\Omega_S\Omega_R(k)\oplus \Omega_S\Omega^{-1}_R(k)\oplus\Omega_S^{-1}\Omega_R(k)\oplus \Omega_S^{-1}\Omega^{-1}_R(k).$$
Then the core series of $M$ is transcendental.
\end{theorem}
We end up this paper with questions on a more general class of series, that is, the $D$-finite series: do modules in the above class have $D$-finite core series, and is a core series $D$-finite in general?

\section{Representation theory of $k$-objects}
In this section, we study the representation theory of $k$-objects, i.e. finitely generated $k[T]$-module supported on the maximal ideal $(T)$. The concepts of the tensor product and the representation ring of $k$-objects are first studied and introduced in \cite{HanMonsky}. In this paper, we will recall results from \cite{HanMonsky}, including the length function and tensor product of $k$-objects, and then we will point out the equivalence of $k$-objects and modular representations of cyclic groups. We fix a field $k$ of characteristic $p>0$ in this section. We follow the convention that any nonpositive power generates the unit ideal, for example, if $T \in k[T]$ and $i \leq 0$, then $(T^i)=k[T]$ as a $k[T]$-ideal. The notation $\ell$ refers to the length of a module. For $k[T]$-module supported on $T$, this is the same as the vector space dimension over $k$.

\subsection{$k$-objects and the additive structure on classes of $k$-objects}
\begin{definition}
\begin{enumerate}
\item A \textup{$k$-object} $M$ with respect to $T$ is a finitely generated $k[T]$-module annihilated by a power of $T$.
\item The direct sum of two $k$-objects is the direct sum as $k[T]$-module.

\item $\Gamma$ is the quotient of the free abelian group over symbols $[M]$ where $M$ runs through isomorphic classes of $k$-objects by the relations $[M\oplus N]-[M]-[N]$. 
\item Let $V_i=k[T]/(T^i)$, $\delta_i=[k[T]/(T^i)] \in \Gamma$ be its class in $\Gamma$ for $i\geq 1$.
\item Let $\Gamma_e=\sum_{1 \leq i \leq p^e}\mathbb{Z}\delta_i$.
\end{enumerate}    
\end{definition}
We see $k[T]$ is a PID. From the structure theorem of PID, we deduce:
\begin{proposition}
\begin{enumerate}
\item For each $i$ there is a unique indecomposable $i$-dimensional $k$-object which is $V_i$.
\item $\Gamma$ is a free abelian group with basis $\delta_i,i \geq 1$ and the addition satisfies $[M]+[N]=[M\oplus N]$.
\item Every element in $\Gamma$ can be expressed as $[M]-[N]$, where $M,N$ are two $k$-objects.
\item $\Gamma_e=\oplus_{1 \leq i \leq p^e}\mathbb{Z}\delta_i$. A $k$-object $M$ is annihilated by $T^{p^e}$ if and only if $[M] \in \Gamma_e$.
\end{enumerate}   
\end{proposition}
\begin{definition}
Let $M$ be a $k$-object with respect to $T$. 
\begin{enumerate}
\item Write $M=\oplus_{i \geq 1}(k[T]/T^i)^{e_i}$, then we define $e_{M,T}(i)=e_i$. It is a function $\mathbb{Z}_{>0} \to \mathbb{N}$ whose value at $i$ is the multiplicity of $k[T]/(T^i)$ in $M$.
\item We define $\ell_{M,T}(i): \mathbb{Z} \to \mathbb{Z}$ to be the following function
\begin{equation*}
\ell_{M,T}(i) = \left\{
        \begin{array}{ll}
            0 & \quad i \leq 0 \\
            \ell(M/T^iM) & \quad i \geq 1
        \end{array}
    \right.
\end{equation*}
and call it \textbf{the length function of $M$}. 
\end{enumerate}
We omit $T$ if the action of $T$ on $M$ is clear.
\end{definition}
By definition, $[M]=\sum_{i \geq 1}e_{M}(i)\delta_i$. For $\gamma\in \Gamma$, we define $e_\gamma=e_M-e_N$ and $\ell_\gamma=\ell_M-\ell_N$ where $M,N$ are two $k$-objects satisfying $\gamma=[M]-[N]$. This definition does not depend on the choice of $M,N$.
\begin{proposition}\label[proposition]{5.2 e_M and l_M relation}
Let $M$ be a $k$-object. Then: 
\begin{enumerate}
\item \begin{equation*}
\ell_M(n) = \left\{
        \begin{array}{ll}
            0 & \quad n \leq 0 \\
            e_{M}(1)+2e_{M}(2)+\ldots+ne_{M}(n)+ne_{M}(n+1)+\ldots & \quad n \geq 1.
        \end{array}
    \right.
    \end{equation*}
\item \begin{equation*}
\ell_{M}(n)-\ell_M(n-1) = \left\{
        \begin{array}{ll}
            0 & \quad n \leq 0\\
            e_{M}(n)+e_{M}(n+1)+\ldots & \quad n \geq 1.
        \end{array}
    \right.
\end{equation*}
\item for any $n \geq 1$,
\begin{equation*}
e_{M}(n)=\ell_M(n)-\ell_M(n-1)-(\ell_{M}(n+1)-\ell_M(n))\\
=2\ell_M(n)-\ell_M(n+1)-\ell_M(n-1).
\end{equation*}
In particular, the structure of $M$ is determined by its length function.
\end{enumerate}
\end{proposition}
\begin{proof}
(1) follows from the definition of $\ell_M(n)$ and $e_{M}(n)$; (2) and (3) follow from (1).
\end{proof}
The length function of a $k$-object is usually easier to compute than the multiplicity of each indecomposable components, and we can recover multiplicities from its length function. So, \textit{the role of the length function of a $k$-object is similar to the role of the character of an ordinary representation.}

We characterize all length functions of $k$-objects.
\begin{theorem}
A function $\ell:\mathbb{N}_{\geq 1} \to \mathbb{N}_{\geq 1}$ gives the length function of a nonzero $k$-object if and only if
\begin{enumerate}
\item $\ell$ is increasing;
\item $\ell$ is concave, that is, $2\ell(n+1)\geq \ell(n)+\ell(n+2)$ for $n \geq 0$ where we denote $\ell(0)=0$;
\item $\ell$ is eventually a constant.
\end{enumerate}
\end{theorem}
\begin{proof}
If $\ell=\ell_M$, it is easy to see $\ell$ is increasing. $\ell$ is convex since $T^iM/T^{i+1}M$ surjects onto $T^{i+1}M/T^{i+2}M$ for any $i \geq 0$. If $T^iM=0$, we see $\ell(j)=\ell(i)$ for any $j \geq i$. So $\ell$ satisfies (1) (2) (3). Conversely if $\ell$ satisfies (1) (2) (3), then we can solve $e(n)=2\ell(n)-\ell(n+1)-\ell(n-1),n \geq 2$ and $e(1)=2\ell(1)-\ell(2)$ from $\ell(n)$. They are nonnegative by convexity of $\ell$. Since $\ell(n)$ is eventually constant, $e(n)$ is eventually zero. So $M=\oplus \delta_n^{e_n}$ is a finite direct sum, so it is a $k$-object. Now it suffices to verify $\sum_i \min\{n,i\}e_i=\ell(n)$, which is straightforward.   
\end{proof}
\subsection{Multilinear forms on $\Gamma$}
In this subsection, we introduce an $s$-fold multilinear map on $\Gamma$ associated with a polynomial $\phi$ of $s$-variables over $k$ without a constant term.

Since $\Gamma$ is the free abelian group with basis $\delta_i, i \geq 1$, to define a multilinear map $\Gamma^s \to \Gamma$, we only need to define it on tuples of basis element $\delta_i$'s. For any $\phi \in k[T_1,T_2,\ldots,T_s]$ without a constant term, we define the multilinear map $B_\phi: \Gamma^s \to \Gamma$ by specifying $B_\phi(\delta_{t_1},\ldots,\delta_{t_s})$ as follows: $$B_\phi(\delta_{t_1},\ldots,\delta_{t_s})=k[T_1,T_2,\ldots,T_s]/(T_1^{t_1},T_2^{t_2},\ldots,T_s^{t_s})$$ 
as a $k$-object with respect to $T=\phi \in k[T_1,\ldots,T_s]$. By convention on the nonpositive power, $B_\phi(\delta_{t_1},\ldots,\delta_{t_s})=0$ if there is $t_i \leq 0$. Here $T$ acts nilpotently since $\phi$ has no constant term. This gives the multilinear form $B_\phi$. The importance of this definition is reflected in the following proposition on $s$-fold tensor product:
\begin{proposition}\label[proposition]{5.2 multilinear form property}
Let $M_i$ be a $k$-object with respect to $f_i$ for $1 \leq i \leq s$. Take $\phi \in k[T_1,T_2,\ldots,T_s]$ and define $B_\phi$ as above. Let $\otimes_k M_i=(\otimes_k)_{1 \leq i \leq s} M_i$ be the $s$-fold tensor product of all $M_i$'s. Then $\phi(\underline{f})$ acts on $ \otimes_k M_i$ which makes $\otimes_k M_i$ a $k$-object with respect to $\phi(\underline{f})$, and as a $k$-object, $[\otimes_k M_i]=B_\phi([M_1],\ldots,[M_s])$.  
\end{proposition}
This is proved in \cite[Theorem 5.7]{arxivAnalysis}. Denote the coefficient of the bilinear form 
$B_\phi$ by $$B_\phi(\mathbf{t},r)=e_{k[T_1,T_2,\ldots,T_s]/(T_1^{t_1},T_2^{t_2},\ldots,T_s^{t_s}),\phi}(r).$$ 
That is, we also view $B_\phi$ as a map $\mathbb{Z}^s\times \mathbb{Z}_{>0} \to \mathbb{Z}$ by abusing the notation. Define
$$D_\phi(\mathbf{t},r)=\ell_{k[T_1,T_2,\ldots,T_s]/(T_1^{t_1},T_2^{t_2},\ldots,T_s^{t_s}),\phi}(r)=\ell(k[T_1,T_2,\ldots,T_s]/(T_1^{t_1},T_2^{t_2},\ldots,T_s^{t_s},\phi^r)).$$ 
From the definition we see for any $\mathbf{t} \in \mathbb{Z}^s$, $B_\phi(\delta_{t_1},\ldots,\delta_{t_s})=\sum_{r \geq 1} B_\phi(\mathbf{t},r)\delta_r$ and if 
$r \geq 1$, $B_\phi(\mathbf{t},r)=2D_\phi(\mathbf{t},r)-D_\phi(\mathbf{t},r+1)-D_\phi(\mathbf{t},r-1)$.
\begin{corollary}[\cite{arxivAnalysis}, Proposition 5.8]
Let $M_i$ be a $k$-object with respect to $f_i$ for $1 \leq i \leq s$. Take $\phi \in k[T_1,\ldots,T_s]$. Then
\begin{enumerate}
\item $e_{B_\phi([M_1],\ldots,[M_s]),\phi(\underline{f})}(r)=\sum_{\mathbf{t} \geq 1}\prod_{1 \leq i \leq s}e_{M_i,f_i}(t_i) B_\phi(\mathbf{t},r)$.
\item $\ell_{B_\phi([M_1],\ldots,[M_s]),\phi(\underline{f})}(r)=\sum_{\mathbf{t} \geq 1}\prod_{1 \leq i \leq s}e_{M_i,f_i}(t_i) D_\phi(\mathbf{t},r)$.
\end{enumerate}
\end{corollary}

\subsection{Tensor products and modular representations of cyclic groups}
In this subsection, we consider the tensor product of $k$-objects, which gives a ring structure on $\Gamma$. This tensor product comes from the bilinear form associated with the polynomial $T_1+T_2$ introduced in last subsection. We will denote this tensor product by $\tilde{\otimes}_k$ temporarily; later we will show this is equivalent to the tensor product of modular representations, and use the terminology $\otimes_k$ or $\otimes$ again.   
\begin{definition}
The tensor product of two $k$-objects $M,N$ over $k$, denoted by $M\tilde\otimes_kN$, is a $k$-object with underlying space $M\otimes_kN$ and $T$-action given by $T(m\otimes n)=Tm\otimes n+m\otimes Tn$. 
\end{definition}
\begin{proposition}
\begin{enumerate}
\item $[M\tilde\otimes_kN]=B_{T_1+T_2}([M],[N])$.
\item $\Gamma$ is a unital commutative ring with unit $[k]=\delta_1$, addition $[M]+[N]=[M\oplus N]$ and multiplication $[M][N]=[M\tilde\otimes_k N]$.
\item $\Gamma_e$ is the subring of $\Gamma$ generated by classes of $k$-objects annihilated by $T^{p^e}$.
\end{enumerate}
\end{proposition}
\begin{proof}
(1) is true by definition. (2) is proved in \cite{HanMonsky}. (3) is true since $T^{p^e}(m\otimes n)=T^{p^e}m\otimes n+m\otimes T^{p^e}n$.    
\end{proof}
The set of $k$-objects annihilated by $T^{p^e}$ corresponds to modular representations of the cyclic $p$-group $\mathbb{Z}/p^e\mathbb{Z}$. Actually, if $G=\langle g \rangle=\mathbb{Z}/p^e\mathbb{Z}$, then $k[G]=k[g]/(g^{p^e}-1)\cong k[T]/(T^{p^e})$ via the isomorphism $g \to 1+T$. In this sense, a $G$-representation $V$ is a $k$-object with respect to $T=g-1$ and $G$ acts on a $k$-object via $gx=(T+1)x$.
\begin{lemma}
Let $M,N$ be two $k$-objects annihilated by $T^{p^e}$ and view them as $G$-representations via $g=1+T$. Then the $T$-action on the tensor product $M\otimes_k N$ as $G$-representations is given by
$$T(m\otimes n)=Tm\otimes Tn+Tm\otimes n+m\otimes Tn.$$
In other words, the $G$-representation tensor product satisfies $[M\otimes_k N]=B_{T_1T_2+T_1+T_2}([M],[N])$.
\end{lemma}
\begin{proof}
By definition $g(m\otimes n)=gm\otimes gn$, which means $(1+T)(m\otimes n)=(1+T)m\otimes (1+T)n$, so the equality holds.   
\end{proof}
\begin{lemma}\label[lemma]{2.11}
Let $\phi \in k[T_1,T_2,\ldots,T_s]$ be a polynomial, $\phi=a_1T_1+\ldots+a_sT_s+\phi',\phi' \in (T_1,\ldots,T_s)^2$ and $a_i \neq 0$ for all $i$. Then $D_\phi(\mathbf{t},r)=D_{T_1+\ldots+T_s}(\mathbf{t},r)$ and $B_\phi(\mathbf{t},r)=B_{T_1+\ldots+T_s}(\mathbf{t},r)$ for all $(\mathbf{t},r) \in \mathbb{N}^{s+1}_{\geq 1}$.   
\end{lemma}
\begin{proof}
Since $B_\phi$ depends on $D_\phi$, it suffices to prove the equality for $D_\phi$. We need to show
$$\ell(k[T_1,\ldots,T_s]/(T_1^{t_1},\ldots,T_s^{t_s},\phi^r))=\ell(k[T_1,\ldots,T_s]/(T_1^{t_1},\ldots,T_s^{t_s},(T_1+\ldots+T_s)^r)).$$
Both sides are annihilated by a power of $(T_1,\ldots,T_s)$, so both lengths stay the same after completion, and it suffices to prove
$$\ell(k[[T_1,\ldots,T_s]]/(T_1^{t_1},\ldots,T_s^{t_s},\phi^r))=\ell(k[[T_1,\ldots,T_s]]/(T_1^{t_1},\ldots,T_s^{t_s},(T_1+\ldots+T_s)^r)).$$
Now by assumption on $\phi$ we can write
$$\phi=a_1T_1(1+f_1(T_1,\ldots,T_s))+a_2T_2(1+f_2(T_2,\ldots,T_s))+\ldots+a_sT_s(1+f_s(T_s))$$
for some polynomials $f_1,\ldots,f_s$ without constant terms. Note that there is an automorphism $\Phi$ of $k[[T_1,\ldots,T_s]]$ that is identity on $k$, maps $(T_1,\ldots,T_s)$ to itself (hence is continuous with respect to the adic topology) and sends $T_i$ to $T_i(1+f_i)$. 
\begin{align*}
\ell(k[[T_1,\ldots,T_s]]/(T_1^{t_1},\ldots,T_s^{t_s},\phi^r))=\\
\ell(k[[T_1,\ldots,T_s]]/((T_1(1+f_1))^{t_1},\ldots,(T_s(1+f_s))^{t_s},\phi^r))=\\
\ell(k[[T_1,\ldots,T_s]]/(T_1^{t_1},\ldots,T_s^{t_s},(a_1T_1+\ldots+a_sT_s)^r))=\\
\ell(k[[T_1,\ldots,T_s]]/(T_1^{t_1},\ldots,T_s^{t_s},(T_1+\ldots+T_s)^r)).
\end{align*}
The first equality holds since $1+f_i$ is a unit in $k[[T_1,\ldots,T_s]]$, the second and third equalities hold by applying $\Phi^{-1}$ and a linear transformation.
\end{proof}
\begin{theorem}\label{2.12}
The tensor products of two $k$-objects as $k$-objects and $G$-representations are isomorphic. Therefore, $\Gamma_e$ is the representation ring of $\mathbb{Z}/p^e\mathbb{Z}$ over a field of characteristic $p$.   
\end{theorem}
\begin{proof}
The tensor product of $k$-objects is given by $B_{T_1+T_2}$ and the tensor product of $G$-representations is given by $B_{T_1T_2+T_1+T_2}$. By \Cref{2.11}, $B_{T_1+T_2}(t_1,t_2,t_3)=B_{T_1T_2+T_1+T_2}(t_1,t_2,t_3)$ and $B_\phi:\Gamma\times\Gamma \to \Gamma$ is determined by the coefficients $B_\phi(t_1,t_2,t_3)$. Thus $B_{T_1+T_2}=B_{T_1T_2+T_1+T_2}$, and we are done.  
\end{proof}
The above theorem tells us that we can discard the notation $\tilde{\otimes}_k$ and say $M\otimes_k N$ is also the tensor product of $M,N$ as $k$-objects.
\begin{remark}
For a non-cyclic group $G$, the analogous statement for \Cref{2.12} fails. We consider the following example.

Let $G=\mathbb{Z}/3\times\mathbb{Z}/3$ over a field $k$ of characteristic $3$. We see $R=kG=k[s,t]/(s^3,t^3)$. Let $R_1=k[s_1,t_1]/(s_1^3,t_1^3),R_2=k[s_2,t_2]/(s_2^3,t_2^3),R_3=R_1\otimes_kR_2=k[s_1,t_1,s_2,t_2]/(s_1^3,t_1^3,s_2^3,t_2^3)$. Let $M=k[s,t]/(s^3,t^3,s+t)$ be an $R$-module. Let $M_1=k[s_1,t_1]/(s_1^3,t_1^3,s_1+t_1),M_2=k[s_2,t_2]/(s_2^3,t_2^3,s_2+t_2)$. Let $\varphi:R \to R_3,s \to s_1+s_2,t \to t_1+t_2$, $\psi:R \to R_3,s \to s_1+s_2+s_1s_2,t \to t_1+t_2+t_1t_2$. We see $M_1\otimes_kM_2$ is an $R_3=R_1\otimes_kR_2$-module. The two tensor products $\otimes_k,\tilde{\otimes}_k$ correspond to the following two pushforwards of $M_3$:
$$M\tilde\otimes_kM=\varphi_*M_3,M\otimes_k M=\psi_*M_3.$$
We compute
\begin{align*}
\varphi_*(M_3)\otimes_R R/(s,t)=\varphi_*(M_3/(\varphi(s),\varphi(t))M_3)\\
=\varphi_*k[s_1,t_1,s_2,t_2]/(s_1^3,t_1^3,s_2^3,t_2^3,s_1+t_1,s_2+t_2,s_1+s_2,t_1+t_2)\\
\overset{s_1 \to -s_2,t_1\to -t_2}{=}\varphi_*k[s_2,t_2]/(s_2^3,t_2^3,s_2+t_2).
\end{align*}
Since $\varphi_*$ does not change the length, this module has length $3$.

Note that $(1+s_2)^{-1}$ and $(1+t_2)^{-1}$ make sense in $R_2$, so
\begin{align*}
\psi_*(M_3)\otimes_R R/(s,t)=\psi_*(M_3/(\psi(s),\psi(t))M_3)\\
=\psi_*k[s_1,t_1,s_2,t_2]/(s_1^3,t_1^3,s_2^3,t_2^3,s_1+t_1,s_2+t_2,s_1+s_2+s_1s_2,t_1+t_2+t_1t_2)\\
\overset{s_1 \to -s_2(1+s_2)^{-1},t_1\to -t_2(1+t_2)^{-1}}{=}\psi_*k[s_2,t_2]/(s_2^3,t_2^3,s_2+t_2,s_2(1+s_2)^{-1}+t_2(1+t_2)^{-1})\\
=\psi_*k[s_2,t_2]/(s_2^3,t_2^3,s_2+t_2,s_2(1+t_2)+t_2(1+s_2))\\
=\psi_*k[s_2,t_2]/(s_2^3,t_2^3,s_2+t_2,s_2t_2).
\end{align*}
This module has length $2$. Therefore, the two tensor products $\otimes_k$ and $\tilde{\otimes}_k$ are distinct.
\end{remark}
\subsection{Operations on representations}
In this subsection, we recall 3 operations on representations. Let $H \subset G$ be a subgroup.
\begin{enumerate}
\item If $V$ is a $G$-representation, the restriction on $H$ gives an $H$-representation, denoted by $\operatorname{Res}^H_GV$.
\item If $W$ is an $H$-representation, $\operatorname{Ind}^G_HW=kG\otimes_{kH}W$ is the induction of $W$, which is a $G$-representation. We also denote it by $W\uparrow_H^G$.
\item If $K \lhd G$ and $V$ is an $H=G/K$-representation, we view it as a $G$-representation via $gx=(g+K)x$. The $G$-representation we get is called the inflation of $V$, denoted by $\operatorname{Inf}^G_HV$. If there is no confusion, we may also say $V$ is a $G$-representation by abusing the notation.
\end{enumerate}
We omit the indices $G,H$ if they are clear from context.

Let $q$ be a power of $p$, $G=\mathbb{Z}/p^eq\mathbb{Z}$ and $H=\mathbb{Z}/p^e\mathbb{Z}$, then there is a unique injection $H \hookrightarrow{}G$ and there is a unique surjection $G \twoheadrightarrow H$. The first map gives the restriction and induction, and the second map gives the inflation. Let $V_a$ be the $a$-dimensional indecomposable $k$-object.
\begin{enumerate}
\item If $a \leq p^e$, view $V_a$ as an $H$-representation, we have $\operatorname{Ind}^G_HV_a=V_{aq}$.
\item If $a \leq p^eq$ and $a=bq+r$, $0 \leq r \leq q-1$, view $V_a$ as a $G$-representation, we have $\operatorname{Res}^H_GV_a=(q-r)V_{b}+rV_{b+1}$.
\item If $a \leq p^e$, view $V_a$ as an $H$-representation, we have $\operatorname{Inf}^G_HV_a=V_{a}$.
\end{enumerate}
We also say the above equalities hold when $V_a$ is replaced with its class $\delta_a$, so $\operatorname{Ind}^G_H,\operatorname{Res}^H_G,\operatorname{Inf}^G_H$ become maps between $\Gamma_e$'s.
\subsection{$\Gamma$ and axioms for a representation ring}
From \cite{BensonBanach}, we see that to study the asymptotic behavior of tensor products of representations, we can embed representation rings into a Banach algebra and study the spectrum of elements in the Banach algebra. When the Banach algebra is semisimple, we can recover the ring structure and spectrums from the Gelfand transformation of completion. This is the case where $G=\mathbb{Z}/p^e\mathbb{Z}$ since $\Gamma_e$ is semisimple. 

In \cite{BensonBanach}, Benson gave an axiomatic definition of a representation ring, that is, a ring with a free $\mathbb{Z}$-basis $x_i,i \in \mathfrak{I}$ such that the product $x_ix_j=\sum_{i,j,k}c_{i,j,k}x_k$ for some $c_{i,j,k}\in\mathbb{N}$ and for each pair $(i,j)$, this is a finite sum, plus some axioms on the dual, the trivial representation $\mathds{1}$ and the regular representation $\rho$. The representation rings of finite groups are motivating examples of a general representation ring. 

Now we compare the ring $\Gamma$ with Benson's axioms. We see $\Gamma$ is not a representation ring, but it satisfies most of the axioms.
\begin{enumerate}
\item $\Gamma_e$'s satisfy the axioms for representation rings since they are representation rings of groups. They form a filtration of $\Gamma$: $\Gamma_1 \subset \Gamma_2 \subset \ldots \subset \Gamma$, and $\Gamma=\cup_e\Gamma_e$.
\item The maps $\Gamma_e \to \Gamma$ for different $e$'s are compatible with inflations.
\item The dimension function is compatible with the embeddings $\Gamma_e \subset \Gamma_{e'}$. Thus the dimension function can be defined on $\Gamma$.
\item $\Gamma$ satisfies all the axioms on duals, $\mathds{1}$, and dimension. This is true since all these axioms are verified on subrings, $\Gamma_e \to \Gamma$ is an embedding such that $\Gamma=\cup_e\Gamma_e$. However, there is no element $\rho \in \Gamma$ satisfying the axiom on regular representations for $\Gamma$. What we have is a sequence of elements $\rho_e$ such that $\rho_e x=\dim x\cdot \rho_e$ for any $x \in \Gamma_e$, and $\Gamma_e$'s form a filtration of $\Gamma$.
\end{enumerate}
The fact (3) leads to the following result.
\begin{theorem}
Consider the norm on $\Gamma$ defined by the dimension, that is, a norm satisfying $\|\sum a_i\delta_i\|=\sum |a_i|\dim \delta_i=\sum |a_i|i$. Then the completion $\hat{\Gamma}$ of $\Gamma$ is a Banach $*$-algebra with product induced from $\Gamma$ and $*$ being the identity map.    
\end{theorem}
\begin{proof}
To prove $\hat{\Gamma}$ has a product structure, it suffices to prove $\|xy\|\leq \|x\|\|y\|$ for $x,y \in \Gamma$, which can be verified on each $\Gamma_e$. We also observe that for $a \leq p^e$, $\delta_a^*=\delta_a$ is independent of choice of $e$, so the claim for $*$ also holds.   
\end{proof}

\section{The real algebra $\mathcal{F}$}
In this section, we introduce a real algebra $\mathcal{F}$ of real-valued functions whose multiplication is given by integration along a kernel function, and embed $\Gamma_e$ into $\mathcal{F}$. The nature of $\Gamma$ and $\mathcal{F}$ are different: the embeddings of different $\Gamma_e$'s into $\Gamma$ are compatible with inflations of representations while embeddings into $\mathcal{F}$ are compatible with induction of representations composed with multiplication by $1/p$. We also show that $\mathcal{F}$ is a normed algebra.

\subsection{The definitions of $\mathcal{F}$ and $*$-product}
\begin{definition}\label[definition]{3.1}
We define $\mathcal{F}^+$ to be the space of the following continuous functions $f$:
\begin{enumerate}
\item $f((-\infty,0])=0$.
\item $f|_{[1,\infty)}$ is a constant.
\item $f|_{[0,\infty)}$ is concave and increasing.
\item $f'_{\pm}$ is bounded, or equivalently, $f'_+(0)<\infty$.
\end{enumerate}
Let $\mathcal{F}$ be the real vector space spanned by $\mathcal{F}^+$ in the space of all functions $\mathbb{R} \to \mathbb{R}$. Since $\mathcal{F}^+$ is a cone in the space of real functions, we have $\mathcal{F}=\mathcal{F}^+-\mathcal{F}^+$.
\end{definition}
\begin{proposition}
$\mathcal{F}$ is exactly the space of Lipschitz continuous functions $f$ such that $f((-\infty,0])=0$, $f|_{[1,\infty)}$ is a constant and $f'$ is of bounded variation on its domain.    
\end{proposition}
\begin{proof}
Any function $f \in \mathcal{F}$ satisfies the above conditions. Conversely if $f$ satisfies the above conditions, we will prove $f \in \mathcal{F}=\mathcal{F}^+-\mathcal{F}^+$. Let $f'$ be its pointwise derivative. Since $f$ is Lipschitz, so it is of bounded variation, $f'$ exists almost everywhere and the fundamental theorem of Calculus $f(x)=f(x)-f(0)=\int^x_0f'(t)dt$ holds. Since $f'$ is of bounded variation, we can write $f'=f'_+-f'_-$ where $f'_\pm$ are increasing. Then $f'_\pm$ are bounded, so we can replace $f'_\pm$ by $f'_\pm-C$ to assume they all have negative values on $[0,1]$. We may also assume they have $0$ values outside $[0,1]$. Let $f_\pm=\int_0^x f'_\pm(t)dt$, we see $f=f_+-f_-=-f_--(-f_+)$ and $f_+,f_-$ are bounded, convex, decreasing, locally constant outside $[0,1]$. So $-f_\pm \in \mathcal{F^+}$ and $f \in \mathcal{F}^+-\mathcal{F}^+$.  
\end{proof}
\begin{proposition}\label[proposition]{3.3}
Let $K$ be a continuous function from an open set containing $[0,1]^2 $ to $\mathbb{R}$ such that $K(0,\cdot)=K(\cdot,0)=0$. Then for any $\alpha_1,\alpha_2 \in \mathcal{F}^+$, the Riemann-Stieltjes integral
$$\int_{[0,1^+]^2}K(t_1,t_2)d\alpha'_1(t_1)d\alpha'_2(t_2)$$
is well-defined. Here the integral of the continuous function $K$ is taken with respect to any monotone extension of $\alpha'_1,\alpha'_2$, and the value does not depend on the choice of the extension.
\end{proposition}
\begin{proof}
This is proved by \cite[Proposition 2.22 and Remark 2.23]{arxivAnalysis}.    
\end{proof}
We reintroduce our function $D_\phi$ from another point of view. It is a special case of the multivariate $h$-function introduced in \cite[Definition 3.1]{arxivAnalysis}
\begin{definition}\label[definition]{3.4}
Let $\phi \in k[T_1,T_2]$ be a polynomial without constant term. The kernel function of $\phi$, denoted by $D_\phi$, is the following function $\mathbb{R}^3 \to \mathbb{R}$:
$$(t_1,t_2,t_3) \to \lim_{e \to \infty}\frac{\ell(k[T_1,T_2]/(T_1^{\lceil t_1p^e \rceil},T_2^{\lceil t_2p^e \rceil},\phi^{\lceil t_3p^e\rceil}))}{p^{2e}}.$$
Here the existence is proved in \cite[Proposition 3.5]{arxivAnalysis}. When $t_i \leq 0$ for some $i$, we denote $(T_1^{\lceil t_1p^e \rceil},T_2^{\lceil t_2p^e \rceil},\phi^{\lceil t_3p^e \rceil})=k[T_1,T_2]$ by convention, so in this case $D_\phi(t_1,t_2,t_3)=0$. We omit $\phi$ when $\phi=T_1+T_2$.
\end{definition}
\begin{remark}
When $t_1,t_2,t_3\in\mathbb{Z}_{>1}$, we have
$$D_\phi(t_1,t_2,t_3)=\ell(k[T_1,T_2]/(T_1^{t_1},T_2^{t_2},\phi^{t_3})),$$
so it coincides with the definition in subsection 2.2.
\end{remark}
\begin{remark}\label[remark]{3.6}
Properties of $D=D_{T_1+T_2}$ are explored in \cite[Section 3 and Section 4]{arxivAnalysis}. We list some of them here:
\begin{enumerate}
\item \cite[Proposition 3.3]{arxivAnalysis} $D$ is a continuous function on $\mathbb{R}^3$. It is concave in each variable on $[0,\infty)$.
\item \cite[Proof of Lemma 4.45]{arxivAnalysis} $\frac{\partial}{\partial t^+}D(t_1,t_2,t)|_{t=0}=\min\{t_1,t_2\}$.
\item \cite[Proposition 4.3(1)]{arxivAnalysis} $D(t_1,t_2,t_3)$ is invariant under any permutation of $t_1,t_2,t_3$.
\item \cite[Proposition 4.3(3)]{arxivAnalysis} $D(t_1,t_2,t)=0$ if $t \leq 0$ and $D(t_1,t_2,1)=t_1t_2$ if $t_1,t_2 \in [0,1]$.
\item \cite[Proposition 4.3(4)]{arxivAnalysis} $D(p^ea,p^eb,p^ec)=p^{2e}D(a,b,c)$ for any $e\in\mathbb{N}$, $a,b,c \in \mathbb{R}$.
\end{enumerate}
\end{remark}
\begin{definition}
For $f_1,f_2 \in \mathcal{F}$, define their product to be the real function $[0,1] \to \mathbb{R}$:
$$f(t)=f_1*f_2(t)=\int_{[0,1^+]^2}D(t_1,t_2,t)d(-f'_1(t_1))d(-f'_2(t_2)).$$
It is well-defined as a function by \Cref{3.3} and bilinearity.
\end{definition}
\subsection{The associativity of $*$-product}
In this subsection, we will prove that $*$ is a binary operation on $\mathcal{F}$ which is associative. Therefore, $\mathcal{F}$ has a real algebra structure.
\begin{lemma}\label[lemma]{3.8}
Suppose $f,g \in \mathcal{F}^+$, then $f*g \in \mathcal{F}^+$. Therefore, $*$ is a map $\mathcal{F}\times\mathcal{F}\to \mathcal{F}$. Also for $f,g \in \mathcal{F}$, $f*g=g*f$. 
\end{lemma}
\begin{proof}
We check the conditions in \Cref{3.1}.

(1) If $t \leq 0$, then $D(t_1,t_2,t)=0$, so its integral is $0$.

(2) If $t \geq 1$,
\begin{align*}
f*g(t)=\int_{[0,1^+]^2}D(t_1,t_2,1)d(-f_1'(t_1))d(-f_2'(t_2))\\
=\int_{[0,1^+]^2}t_1t_2d(-f_1'(t_1))d(-f_2'(t_2))\\
=\int_{[0,1^+]^2}f_1'(t_1)f_2'(t_2)dt_1dt_2=f_1(1)f_2(1),
\end{align*}
which is a constant.

(3)We see
$$\frac{d}{dt^+}(f*g(0))=\int_{[0,1^+]}\frac{\partial}{\partial t^+}D(t_1,t_2,0)d(-f'(t_1))d(-g'(t_2)).$$
But $\frac{\partial}{\partial t^+}D(t_1,t_2,0)=\min\{t_1,t_2\}$ by \Cref{3.6} and $-f'(t_1),-g'(t_2)$ are increasing, so it suffices to prove
$$\int_{[0^+,1^+]^2}t_1d(-f'(t_1))d(-g'(t_2))=f(1)(-g'_{+}(0))<\infty.$$

(4)For $f,g \in \mathcal{F}^+$, $-f'(t_1),-g'(t_2)$ are increasing, so the Riemann-Stieltjes integral of positive functions with respect to $-f'\cdot-g'$ is still nonnegative. From the fact that $t \to D(t_1,t_2,t)$ is increasing and concave, so is its integral, which is $f*g$.

So $f*g \in \mathcal{F}^+$, that is, $*$ maps $\mathcal{F}^+\times \mathcal{F}^+$ to $\mathcal{F}^+$. By bilinearity, $*$ maps $\mathcal{F}\times \mathcal{F}$ to $\mathcal{F}$. The commutativity of $*$ comes from symmetry of $D$ with respect to $t_1,t_2$.
\end{proof}
\begin{lemma}\label[lemma]{3.9}
Let $f_e,g_e$ be a sequence of elements in $\mathcal{F}^+$ with uniformly bounded derivatives. Then the sequence $f_e*g_e$ has uniformly bounded derivatives.    
\end{lemma}
\begin{proof}
By \Cref{3.8}, $f_e*g_e \in \mathcal{F}^+$. So it suffices to prove that the right derivative of $f_e*g_e$ at $0$ is uniformly bounded. Using the same proof of \Cref{3.8} (3), we see
$$\frac{d}{dt^+}(f_e*g_e(0))\leq f_e(1)(-g'_{e,+}(0)).$$
But we have $f_e(1)\leq f'_{e,+}(0)$, so it is uniformly bounded in terms of $e$.
\end{proof}
\begin{lemma}\label[lemma]{3.10}
Let $f_e,g_e$ be a sequence of elements in $\mathcal{F}^+$ with uniformly bounded derivatives. Assume $f_e \to f \in \mathcal{F}$ and $g_e \to g \in \mathcal{F}$ pointwisely. Then $f_e*g_e \to f*g$ pointwisely.     
\end{lemma}
\begin{proof}
The proof is the same as \cite[Lemma 5.17]{arxivAnalysis}. We use concavity of $f,g,f_e,g_e$ and \cite[Lemma 4.43]{arxivAnalysis} to deduce that $f'_e \to f'$ and $g'_e \to g'$ outside countably many points where $f'$ or $g'$ does not exist or $x=0$. Then apply \cite[Theorem 2.28]{arxivAnalysis} to get the result.   
\end{proof}

\begin{definition}
For $e \geq 1$ and $1 \leq a \leq p^e$, define
$$\alpha_{a,e}=\max\{\min\{t,a/p^e\},0\} \in \mathcal{F}^+.$$
Let $\mathcal{F}_e$ be the $\mathbb{R}$-linear subspace of $\mathcal{F}$ spanned by $\alpha_{a,e}$, $1 \leq a \leq p^e$. This is exactly the space of functions that are piecewise constant on $[(a-1)/p^e,a/p^e]$ for all $1 \leq a \leq p^e$, equal to $0$ on $(-\infty,0]$ and constant on $[1,\infty)$.
\end{definition}
\begin{lemma}
For $f \in \mathcal{F}$, there exists a sequence of functions $f_e \in \mathcal{F}_e$ such that $f_e \xrightarrow[]{e \to \infty} f$ pointwisely, and $f'_{e,\pm}$ is uniformly bounded on $[0,1]$.
\end{lemma}
\begin{proof}
Since $\mathcal{F}=\mathcal{F}^+-\mathcal{F}^+$, we may assume $f \in \mathcal{F}^+$. Let $f_e$ be the unique function which is piecewise linear on $[a/p^e,(a+1)/p^e]$ and coincides with $f$ on $1/p^e\mathbb{Z}$, then $f_e$ satisfies the desired condition.
\end{proof}
In the following part, $\delta_a,a \in \mathbb{N}$ refers to a class of $k$-objects while $\bm{\delta}_a,a \in\mathbb{R}$ refers to the Dirac delta function, which is a generalized function.
\begin{lemma}
Let $n \geq 1$ and $1 \leq a,b \leq p^e$. Then
$$\alpha_{a,e}*\alpha_{b,e}=\frac{1}{p^e}\sum_{1 \leq c \leq p^e} B(a,b,c)\alpha_{c,e}.$$
\end{lemma}
\begin{proof}
Recall that in $\Gamma_e$ we have
$$\delta_a\delta_b=\sum_{1 \leq c \leq p^e} B(a,b,c)\delta_c,V_a\otimes V_b=\oplus V_{c}^{B(a,b,c)}.$$
Let $q=p^{e'}$ be a power of $p$ where $e'$ is a parameter in $\mathbb{N}$. We see $-\alpha''_{a,e}=\bm\delta_{a/p^e}-\bm\delta_0$ and $D(0,*,*)=D(*,0,*)=0$. So
\begin{align*}
\alpha_{a,e}*\alpha_{b,e}(t)=\int_{[0,1^+]^2}D(t_1,t_2,t)\bm\delta_{a/p^e}(t_1)\bm\delta_{b/p^e}(t_2)dt_1dt_2\\
=D(a/p^e,b/p^e,t)=1/p^{2e}D(a,b,p^et)\\
=\lim_{e' \to \infty}\frac{\ell(k[T_1,T_2]/(T_1^{aq},T_2^{bq},(T_1+T_2)^{\lceil p^etq \rceil}))}{p^{2e}q^2}. 
\end{align*}
It suffices to prove
$$\frac{1}{p^e}\sum_{1 \leq c \leq p^e} B(a,b,c)\alpha_{c,e}=\lim_{e' \to \infty}\frac{\ell(k[T_1,T_2]/(T_1^{aq},T_2^{bq},(T_1+T_2)^{\lceil p^etq \rceil}))}{p^{2e}q^2}.$$
We observe that both sides are increasing functions supported on $[0,1]$ and the left side is continuous, so it suffices to verify the equality when $t \in \mathbb{Z}[1/p]$. Moreover when $tq_0 \in \mathbb{Z}$, the right side is constant for $q \geq q_0$, so it suffices to check that when $tq \in \mathbb{Z}$,
$$\frac{1}{p^e}\sum_{1 \leq c \leq p^e} B(a,b,c)\alpha_{c,e}=\frac{\ell(k[T_1,T_2]/(T_1^{aq},T_2^{bq},(T_1+T_2)^{p^etq}))}{p^{2e}q^2}.$$
Let $G=\mathbb{Z}/p^eq\mathbb{Z}$, $H=\mathbb{Z}/p^e\mathbb{Z}$. We have
\begin{align*}
V_{aq}\otimes V_{bq}=\operatorname{Ind}^G_H(V_a)\otimes \operatorname{Ind}^G_H(V_b)=\operatorname{Ind}^G_H(V_a\otimes \operatorname{Res}^H_G\operatorname{Ind}^G_H(V_b))\\
=\operatorname{Ind}^G_H(V_a\otimes V_b)^{\oplus [G:H]}=\operatorname{Ind}^G_H(\oplus_c V_{c}^{B(a,b,c)})^{\oplus [G:H]}=\oplus_c V_{cq}^{qB(a,b,c)}.   
\end{align*}
So in $\Gamma$ we have
$$\delta_{aq}\delta_{bq}=\sum_{1 \leq c \leq p^e} qB(a,b,c)\delta_{cq}.$$
As $k[T]=k[T_1+T_2]$-modules, we have
$$k[T_1,T_2]/(T_1^{aq},T_2^{bq})\cong \oplus_{1 \leq c \leq p^e}k[T]/(T^{cq})^{\oplus qB(a,b,c)}.$$
So
$$\frac{\ell(k[T_1,T_2]/(T_1^{aq},T_2^{bq},(T_1+T_2)^{p^etq}))}{p^{2e}q^2}=\frac{1}{p^{2e}q}\sum_{1 \leq c \leq p^e}B(a,b,c)\ell(k[T]/(T^{p^etq},T^{cq})).$$
Now it suffices to verify
$$\frac{1}{p^eq}\ell(k[T]/(T^{p^etq},T^{cq}))=\alpha_{c,e}(t)=\max\{\min\{t,c/p^e\},0\},$$
which is straightforward.
\end{proof}
\begin{corollary}\label[corollary]{3.14}
Let $Z_e=\sum_{1 \leq a \leq p^e} \mathbb{Z}p^e\alpha_{a,e} \subset \mathcal{F}$. Then the map
$$\varphi_e:\Gamma_e \to Z_e,\delta_a \to p^e\alpha_{a,e}$$
is an isomorphism of abelian groups commuting with products. Therefore, $(Z_e,+,*)$ is an associative and commutative ring isomorphic to $\Gamma_e$. Also, $\mathcal{F}_e=Z_e\otimes_\mathbb{Z}\mathbb{R}=\Gamma_e\otimes_\mathbb{Z}\mathbb{R}$ is an $\mathbb{R}$-algebra.
\end{corollary}
\begin{theorem}\label{3.15}
$*$ is associative on $\mathcal{F}$. Therefore, $(\mathcal{F},+,*)$ is an associative and commutative $\mathbb{R}$-algebra, and $\Gamma_e$ embeds into $\mathcal{F}$ with image $Z_e$.    
\end{theorem}
\begin{proof}
By bilinearity of $*$ and $\mathcal{F}=\mathcal{F}^+-\mathcal{F}^+$, it suffices to prove for $f,g,h \in \mathcal{F}^+$,
$$(f*g)*h=f*(g*h).$$
We choose sequences $f_e,g_e,h_e \in \mathcal{F}_e\cap \mathcal{F}^+$ such that $f_e \to f,g_e \to g,h_e\to h$ pointwisely and the derivatives of $f_e,g_e,h_e$ are uniformly bounded. The associativity on $\mathcal{F}_e\cap \mathcal{F}^+$ leads to
$$(f_e*g_e)*h_e=f_e*(g_e*h_e).$$
Now $f_e*g_e$ and $g_e*h_e$ are families in $\mathcal{F}^+$ with bounded derivatives by \Cref{3.9}. So letting $e \to \infty$, we get the associativity for $f,g,h$ by \Cref{3.10}.
\end{proof}
\subsection{The functional counterpart of the induction and the restriction of representations.}
In this subsection, we express the induction and restriction of representations as maps between $Z_e$ and $Z_{e+1}$. Let $H=\mathbb{Z}/p^e\mathbb{Z} \hookrightarrow{}G=\mathbb{Z}/p^{e+1}\mathbb{Z}$. Here $\operatorname{Ind}=\operatorname{Ind}^G_H$ and $\operatorname{Res}=\operatorname{Res}^H_G$.
\begin{proposition}
Let $R_e$ be the map which sends $f \in \mathcal{F}$ to the unique function $R_e(f) \in \mathcal{F}$ which is piecewise linear on $[a/p^e,(a+1)/p^e]$ and coincides with $f$ on $1/p^e\mathbb{Z}$. Then the following two diagrams commute.
\begin{center}
\begin{tikzcd}
\Gamma_e\arrow[r, "\operatorname{Ind}"] \arrow[d,"\varphi_e"]& \Gamma_{e+1} \arrow[d,"\varphi_{e+1}"] \\
Z_e \arrow[r,"p"]& Z_{e+1}
\end{tikzcd}    
\end{center}
and
\begin{center}
\begin{tikzcd}
\Gamma_{e+1}\arrow[r, "\operatorname{Res}"] \arrow[d,"\varphi_{e+1}"]& \Gamma_e \arrow[d,"\varphi_e"] \\
Z_{e+1} \arrow[r,"R_e"]& Z_e
\end{tikzcd}    
\end{center}
\end{proposition}
\begin{proof}
The induction maps $\delta_a$ to $\delta_{ap}$, thus $\delta_a$ is mapped to $p^{e+1}\alpha_{ap,e+1}$ from both directions, so the commutativity of the first diagram is proved.

Take $0 \leq a \leq p^{e+1}$ and write $a=bp+r$ with $0 \leq b \leq p^e,0 \leq r \leq p-1$. We want to show
$$R_e(\varphi_{e+1}\delta_a)=\varphi_e(\operatorname{Res}(\delta_a))=\varphi_e((p-r)\delta_b+r\delta_{b+1}).$$
In other words,
$$R_e(p^{e+1}\alpha_{a,e+1})=p^e(p-r)\alpha_{b,e}+p^er\alpha_{b+1,e}.$$
Since the image of an element under $\varphi_e$ is already piecewise linear on $[c/p^e,(c+1)/p^e]$, it suffices to show for any integer $0 \leq c \leq p^e$,
$$p^{e+1}\alpha_{a,e+1}(c/p^e)=p^e(p-r)\alpha_{b,e}(c/p^e)+p^er\alpha_{b+1,e}(c/p^e).$$
It is staightforward to verify that for $c \leq b$ both sides are equal to $cp$ and for $c \geq b+1$ both sides are equal to $a$. So the commutativity of the second diagram is proved.
\end{proof}
We see $\lim_{e \to \infty} R_e$ is the identity map on $\mathcal{F}$. The Frobenius reciprocity for functions in $\mathcal{F}$ takes the form
$$\int_{[0,1^+]^2}D(t_1,t_2,t)d(-h'_1(t_1))d(-(ph_2)'(t_2))=p\cdot\int_{[0,1^+]^2}D(t_1,t_2,t)d(-h'_1(t_1))d(-h'_2(t_2)),$$
which is a trivial identity.
\subsection{A norm on the real algebra $\mathcal{F}$}
We get the convergence result and associativity of the $*$-product from pointwise convergence in $\mathcal{F}$, which does not come from any norm. Now we want to introduce a norm on $\mathcal{F}$. 
\begin{proposition}
Let $\|\cdot\|:\mathcal{F}^+ \to \mathbb{R}_{\geq 0},f \to \int_0^1|f'(t)|dt$.
\begin{enumerate}
\item For $f \in \mathcal{F}^+$, $\|f\|=f(1)$.
\item For $f,g \in \mathcal{F}^+$, $\|f+g\|=\|f\|+\|g\|$.
\item For $f,g \in \mathcal{F}^+$, $\|f*g\|=\|f\|\cdot\|g\|$.
\end{enumerate}
\end{proposition}
\begin{proof}
(1): We have $\|f\|=\int_0^1|f'(t)|dt=\int_0^1f'(t)dt=f(1)$.

(2): This is true by (1).

(3): Since $f,g \in \mathcal{F}^+$, $f*g \in \mathcal{F}^+$, so
\begin{align*}
\|f*g\|=f*g(1)=\int_{[0,1^+]^2}D(x,y,1)d(-f'(x))d(-g'(y))\\
=\int_{[0,1^+]^2}xyd(-f'(x))d(-g'(y))\\
\overset{\textup{integration by parts}}{=}\int_{[0,1^+]^2}f'(x)g'(y)dxdy\\
=f(1)g(1)=\|f\|\cdot\|g\|.
\end{align*}
\end{proof}
\begin{definition}
Let $\|\cdot\|:\mathcal{F}\to\mathbb{R}_{\geq 0}$ be the following map: $\|f\|$ is the infimum of $\|f_1\|+\|f_2\|$ where $f_1,f_2 \in \mathcal{F}^+$ such that $f=f_1-f_2$.    
\end{definition}
\begin{proposition}\label[proposition]{3.19}
\begin{enumerate}
\item The two definitions of $\|\cdot\|$ on $\mathcal{F}^+$ coincide.
\item $\|\cdot\|$ is a norm.
\item For $f,g \in \mathcal{F}$, $\|f*g\|\leq\|f\|\cdot\|g\|$. So $\mathcal{F}$ is a normed algebra.
\end{enumerate}    
\end{proposition}
\begin{proof}
(1) If $f=f_1-f_2$, then $f'=f_1'-f_2'$, so $\int_0^1|f'(t)|dt\leq \int_0^1|f_1'(t)|dt+\int_0^1|f_2'(t)|dt$. Equality holds when $f_1=f,f_2=0$ and in this case $f_1,f_2 \in \mathcal{F}^+$. So the equality can be achieved by two elements $f_1,f_2 \in \mathcal{F}^+$ and the two definitions of $\|\cdot\|$ coincide on $\mathcal{F}^+$.

(2) First, $\|f\|=\|-f\|$ because $f=f_1-f_2$ implies $-f=f_2-f_1$. Then, for $\lambda>0$, $f=f_1-f_2$ implies $\lambda f=\lambda f_1-\lambda f_2$ and $\lambda f_1,\lambda f_2 \in \mathcal{F}^+$, so $\|\lambda f\|=\lambda\|f\|$. Finally, if $f=f_1-f_2,g=f_3-f_4$ for $f_1,f_2,f_3,f_4 \in \mathcal{F}^+$, then $f+g=(f_1+f_3)-(f_2+f_4), f_1+f_3,f_2+f_4 \in \mathcal{F}^+$. In this case
$$\int_0^1|f'_1(t)+f'_3(t)|+|f_2'(t)+f_4'(t)|dt\leq \int_0^1|f'_1(t)|+|f'_2(t)|+|f_3'(t)|+|f_4'(t)|dt.$$
Taking infimum over all $f_1,f_2,f_3,f_4$, we see
$$\int_0^1|f'_1(t)+f'_3(t)|+|f_2'(t)+f_4'(t)|dt\leq \|f\|+\|g\|.$$
Thus $\|f+g\|\leq \|f\|+\|g\|$. So $\|\cdot\|$ is a norm.

(3) We take $\epsilon>0$ and $f_1,f_2,g_1,g_2$ such that $f=f_1-f_2,g=g_1-g_2$, $\|f_1\|+\|f_2\|\leq \|f\|+\epsilon$ and $\|g_1\|+\|g_2\|\leq \|g\|+\epsilon$. We have
$f*g=(f_1-f_2)*(g_1-g_2)=(f_1*f_2+g_1*g_2)-(f_1*g_2+f_2*g_1)$ with $f_1*f_2+g_1*g_2,f_1*g_2+f_2*g_1\in\mathcal{F}^+$. Thus
\begin{align*}
\|f*g\|\leq \|f_1*f_2+g_1*g_2\|+\|f_1*g_2+f_2*g_1\|\\
=\|f_1\|\|f_2\|+\|g_1\|\|g_2\|+\|f_1\|\|g_2\|+\|f_2\|\|g_1\|\\
=(\|f_1\|+\|f_2\|)(\|g_1\|+\|g_2\|)\\
\leq (\|f\|+\epsilon)(\|g\|+\epsilon).
\end{align*}
Letting $\epsilon \to 0$, we get the desired inequality.
\end{proof}
The completion of $\mathcal{F}$ will be a Banach algebra, and $\Gamma_e\otimes_\mathbb{Z}\mathbb{R} \cong Z_e\otimes_\mathbb{Z}\mathbb{R}$ is a Banach subalgebra of finite dimension in its completion. However, at this stage we don't know if $\mathcal{F}$ itself is complete or not.
\begin{question}
What is the completion $\hat{\mathcal{F}}$, and what is the closure of $\cup_{e \geq 0}(Z_e\otimes_\mathbb{Z}\mathbb{R})$ in $\hat{\mathcal{F}}$? 
\end{question}

\subsection{The application to asymptotic behavior of tensor powers}
The product structure of $\mathcal{F}$ allows us to express powers of elements in $Z_e$ in terms of the kernel function $D$. We show this relationship in the following calculations.

Suppose $V=\oplus_{1 \leq i \leq p^e}V_i^{e_i}$ is a $k$-object annihilated by $T^{p^e}$. Its class in $\Gamma_e$ is $\gamma=[V]=\sum_{1 \leq i \leq p^e}e_i\delta_i$ and its image in $\mathcal{F}$ is $f=\varphi_e(\gamma)=p^e\sum_{1 \leq i \leq p^e}e_i\alpha_{i,e}$. Therefore, $\varphi_e(\gamma^m)=f^{*m}$. $f^{*m}$ is a sequence of elements satisfying $f^{*1}=f$ and $f^{*m}=f*f^{*(m-1)}$. Thus we have
$$f^{*m}(t)=f*f^{*(m-1)}(t)=\int_{[0,1^+]^2}D(t_1,t_2,t)d(-f'(t_1))d(-(f^{*(m-1)})'(t_2)).$$
We see $f'=p^e\sum_{1 \leq i \leq p^e}e_i\chi_{[0,a/p^e]}(t),f''=p^e\sum_{1 \leq i \leq p^e}e_i(\bm\delta_0(t)-\bm\delta_{a/p^e}(t))$, and $D(0,*,*)=0$. So
$$\int_{t_1 \in [0,1^+]}D(t_1,t_2,t)d(-f'(t_1))=p^e\sum_{1 \leq i \leq p^e}e_iD(i/p^e,t_2,t).$$
This means
\begin{align*}
f^{*m}(t)=\int_{t_2 \in [0,1^+]}p^e\sum_{1 \leq i \leq p^e}e_iD(i/p^e,t_2,t)d(-(f^{*(m-1)})'(t_2))\\
=\int_{t_2 \in [0,1^+]}p^e\sum_{1 \leq i \leq p^e}e_i(-\frac{\partial^2D}{\partial t_2^2})(i/p^e,t_2,t)f^{*(m-1)}(t_2)dt_2,
\end{align*}
where the last step uses integration by parts. We can write 
$$K(x,t)=p^e\sum_{1 \leq i \leq p^e}e_i(-\frac{\partial^2D}{\partial t_2^2})(i/p^e,x,t),$$ 
then
$$f^{*m}(x)=\int_0^{1^+}K(x,t)f^{*(m-1)}(t)dt.$$
Therefore, $f^{*m}$ is an iteration of a Fredholm integral transform. Here $K(x,t)$ is a generalized function.

If the eigenvalue and eigenfunction of the kernel $K(x,t)$ are clear, then $f^{*m}$ can be expressed in terms of these data. However, these data are hard to compute. On the other hand, we can compute $f^{*m}(x)$ for all $m$ using power series.

We denote $\Phi(\alpha,x)=\sum_{i \geq 1}\alpha^if^{*i}(x)$. Note that since $\max f^{*m}[0,1]\leq \|f^{*m}\|\leq \|f\|^m$, this series converges to a real function for $\alpha$ in a sufficiently small neighbourhood of $0$. Then
$$\Phi(\alpha,x)-\int_0^{1^+}\alpha K(x,t)\Phi(\alpha,t)dt=\alpha f(x).$$
This gives an integral equation of $\Phi(\alpha,x)$, and $f^{*m}(x)$ is the coefficient of $\alpha^m$ when we find the Taylor expansion of $\Phi$ with respect to $\alpha$ near $0$. When this equation can be solved explicitly, we have a way to compute $f^{*m}$ in terms of $m$, and hence $\gamma^m$ is computable. This is the idea behind the proof of \cite[Corollary 8.8]{arxivAnalysis}.
\section{Asymptotic behavior of syzygies and cosyzygies of $K$}
This section is devoted to the asymptotic behavior of syzygies and cosyzygies of the trivial representation in modular representation theory. In this section, we assume $k$ is a field and $R$ is a commutative, associative, finite-dimensional $k$-algebra throughout this section. We assume every $R$-module is finitely generated (hence finite-dimensional over $k$). If $M$ is an $R$-module, $M^*=\operatorname{Hom}_k(M,k)$ is also an $R$-module via premultiplication.
\subsection{Modules over symmetric algebras}
\begin{definition}
$R$ is called a symmetric algebra if $R^*=R$ as an $R$-module.    
\end{definition}
When $R$ is local, then for any module $M$, the submodule $\operatorname{Hom}_R(k,M)=0:_R\mathfrak{m}$ of $M$ is called the socle of $M$, denoted by $\operatorname{soc}(M)$. The number of generators of $M$ and the length of $\operatorname{soc}(M)$ are denoted by $\mu(M)$ and $r(M)$ respectively.
\begin{lemma}
Let $M,N$ be two $R$-modules. Then $\operatorname{Hom}_R(M,N)\cong \operatorname{Hom}_R(N^*,M^*)$, and $\mu(N)=r(N^*)$. Moreover if $R$ is symmetric, then
\begin{enumerate}
\item $\operatorname{Hom}_R(M,R)\cong M^*$. That is, the $R$-dual and $k$-dual coincide.
\item $\operatorname{Hom}_R(M,N)\cong (M\otimes_R N^*)^*$.
\end{enumerate}
\end{lemma}
\begin{proof}
Since $M\cong M^{**}$, $f \to f^*$ gives an isomorphism $\operatorname{Hom}_R(M,N)\cong \operatorname{Hom}_R(N^*,M^*)$. Apply this isomorphism to the case $M=k$, we see $\mu(N^*)=r(N)$, and replace $N$ by $N^*$, we get $\mu(N)=r(N^*)$. If $R$ is symmetric, then $\operatorname{Hom}_R(M,R)=\operatorname{Hom}_R(R^*,M^*)=\operatorname{Hom}_R(R,M^*)=M^*$. So $\operatorname{Hom}_R(M,N)=\operatorname{Hom}_R(N^*,M^*)=\operatorname{Hom}_R(N^*,\operatorname{Hom}_R(M,R))=\operatorname{Hom}_R(N^*\otimes_RM,R)=(M\otimes_R N^*)^*$. 
\end{proof}
We recall the definition of $n$-th syzygy and cosyzygy of a module, denoted by $\Omega^n,\Omega^{-n}$ respectively. In some literature, they are also called Heller shifts.
\begin{definition}
Assume $R$ is local. For $n>0$ and a module $M$, we define $\Omega(M)$ to be the syzygy module of $M$, that is, there is a projective cover $P\to M$ and $\Omega(M)$ is its kernel. We define $\Omega^n(M)=\Omega^{n-1}(\Omega(M))$ for $n>0$. We define $\Omega^{-1}(M)$ to be the cosyzygy module of $M$, that is, there is an injective hull $M\to E$ and $\Omega^{-1}(M)$ is its cokernel. We define $\Omega^{-n}(M)=\Omega^{1-n}(\Omega^{-1}(M))$ for $n>0$. Denote $\Omega^0(M)=M$.
\end{definition}
Let $\beta^R_i(M)$ be the $i$-th Betti number of $M$ with respect to $R$. Let $\gamma^R_i(M)=\sum_{0 \leq j \leq i}(-1)^{i-j}\beta^R_j(M)$ be the $i$-th truncated alternating sum of Betti numbers. We define $\gamma^R_{i}=0$ for $i<0$ by convention. We omit $R$ in the upper index of $\beta^R_i$ or $\gamma^R_i$ if $R$ is clear.
\begin{lemma}\label[lemma]{4.4}
Assume $n \in \mathbb{N}$ and $M$ is an $R$-module.
\begin{enumerate}
\item When $R$ is symmetric, $\Omega^{-n}(M)=(\Omega^n(M^*))^*$.
\item $\ell(\Omega^n(M))=\gamma_{n-1}(M)\cdot \ell(R)+(-1)^n\ell(M)$.
\item When $R$ is symmetric, $\ell(\Omega^{-n}(M))=\gamma_{n-1}(M^*)\cdot \ell(R)+(-1)^n\ell(M^*)$.
\item $\mu(\Omega^n(M))=r(\Omega^{-n}(M^*))=\beta_n(M)$, and when $R$ is symmetric and $n \geq 1$, $\mu(\Omega^{-n}(M))=r(\Omega^n(M^*))=\beta_{n-1}(M^*)$.
\end{enumerate}    
\end{lemma}
\begin{proof}
Let the minimal free resolution of $M$ over $R$ be
$$\cdots F_1 \to F_0 \to M \to 0,$$
then there is an exact sequence
$$0 \to \Omega^n(M)\to F_{n-1}\to\cdots F_1 \to F_0 \to M \to 0.$$
When $R$ is symmetric, finite projective modules and injective modules coincide and are preserved under $^*$. So, if $M^*$ has a minimal free resolution over $R$
$$\cdots G_1 \to G_0 \to M^* \to 0,$$
then $M$ has a minimal injective resolution
$$0 \to M \to G_0^* \to G_1^* \to \cdots,$$
which gives an exact sequence
$$0 \to M \to G_0^* \to G_1^* \to \cdots \to G_{n-1}^* \to \Omega^{-n}(M) \to 0$$
and hence there is an exact sequence
$$0 \to (\Omega^{-n}(M))^*\to G_{n-1}\to\cdots G_1 \to G_0 \to M^* \to 0.$$    
Hence we have $\Omega^{-n}(M)=(\Omega^n(M^*))^*$. This proves (1). Taking alternating sum of the above exact sequences yields (2) and (3). To prove $\mu(\Omega^n(M))=\beta_i(M)$, we observe that $F_{n+1} \to F_n \to \Omega^n(M) \to 0$ is a minimal presentation, so $\mu(\Omega^n(M))=\operatorname{rank}F_n=\beta_n(M).$ To prove $\mu(\Omega^{-n}(M))=\beta_{n-1}(M^*)$, observe that $G_{n-2}^* \to G_{n-1}^* \to \Omega^{-n}(M) \to 0$ is a minimal presentation, so $\mu(\Omega^{-n}(M))=\operatorname{rank}G_{n-1}^*=\beta_{n-1}(M^*).$ The other two equalities follow from (1) and the equality $\mu(N)=r(N^*)$.
\end{proof}
\begin{lemma}
Let $R$ be symmetric, $M$ be an $R$-module without free summand, then $M$ is indecomposable if and only if $\Omega(M)$ is indecomposable.   
\end{lemma}
\begin{proof}
$\Omega$ commutes with direct sums, so $M$ is decomposable implies $\Omega(M)$ is decomposable. If $M$ has no free summands, then $M=\Omega^{-1}(\Omega(M))=\Omega(\Omega(M)^*)^*$, so $\Omega(M)$ is decomposable implies $M$ is decomposable.
\end{proof}
\subsection{Tate's resolution and length computation}

In this subsection, we compute some lengths concerning $\Omega^n(k)$. From now on till the end of this section, we assume the following: $R=kG$ where $G$ is a finite abelian $p$-group, and $k$ is a field of characteristic $p>0$. We assume $G=(\mathbb{Z}/p^{e_1})g_1\times (\mathbb{Z}/p^{e_2})g_2\times\ldots\times(\mathbb{Z}/p^{e_r})g_r$ for some $e_1,e_2,\ldots,e_r,r\geq 1$. Here $r$ is the number of generators of $G$. We see $R=k[g_1,\ldots,g_r]/(g_1^{p^{e_1}}-1,\ldots,g_r^{p^{e_r}}-1)=k[t_1,\ldots,t_r]/(t_1^{p^{e_1}},\ldots,t_r^{p^{e_r}})$ is a complete intersection ring. It is local with maximal ideal $\mathfrak{m}=(t_1,\ldots,t_r)$ and residue field $k$. In this case, every finite $R$-module $M$ satisfies $\ell(M)=\dim_kM$, so we will use these two notions interchangably. Let $I$ be an ideal $(f_1,\ldots,f_r)$ in $\mathfrak{m}$ generated by a regular sequence in $k[t_1,\ldots,t_r]$ such that $t_i^{p^{e_i}}$ lies in the preimage of $\mathfrak{m}I$. Then the minimal resolution of $R/I$ over $R$ is Tate's resolution, see \cite{Tate}. We have
$$\operatorname{Tor}^*_R(R/I,k)\cong k\langle y_1,y_2,\ldots,y_r\rangle[x_1^{(i)},x_2^{(i)},\ldots,x_r^{(i)}|i\geq 1],$$
where $x_1^{(i)}$ refers to the divided power. The homological degree of $y_j$ is $1$ and $x_j^{(i)}$ is $2i$. Since $k$ has characteristic $p$, we can also write
$$\operatorname{Tor}^*_R(R/I,k)\cong k\langle y_1,y_2,\ldots,y_r\rangle[x_1^{(p^e)},x_2^{(p^e)},\ldots,x_r^{(p^e)}|e\geq 0]$$
with $(x_j^{(p^e)})^p=0$. In particular, the length of graded pieces of $\operatorname{Tor}^*_R(R/I,k)$ is the same as that of $k\langle y_1,y_2,\ldots,y_r\rangle[x_1,\ldots,x_2,\ldots,x_r]$ where $x_i$'s are commuting free variables. We denote
$$\beta_M(z)=\sum_{i \geq 0}\beta^R_i(M)z^i,\gamma_M(z)=\sum_{i \geq 0}\gamma^R_i(M)z^i.$$
Then
$$\beta_{R/I}(z)=(1+z)^r(1-z^2)^{-r}=(1-z)^{-r}$$
and
$$\gamma_{R/I}(z)=\beta_{R/I}(z)(1+z)^{-1}=(1-z)^{-r}(1+z)^{-1},$$
so we have
$$\beta_i(R/I)={r+i-1\choose i},\gamma_i(R/I)=\sum_{0 \leq j \leq i}(-1)^{i-j}{r+j-1\choose j}.$$
In particular, the above result holds for $k=R/(t_1,\ldots,t_r)$. \Cref{4.6} deals with the asymptotic behavior of $\beta_i,\gamma_i$ in terms of $i$.
\begin{proposition}\label[proposition]{4.6}
We fix $r \geq 2$.
\begin{enumerate}
\item $\beta_i$ is a polynomial in $i$ with leading term $\frac{1}{(r-1)!}i^{r-1}$ for $i \geq 0$.
\item $\gamma_i$ is a quasi-polynomial of period $2$ for $i>r+1$. That is, there are two polynomials $P_1,P_2$ such that $\gamma_i=P_1(i)$ for odd $i>r+1$ and $\gamma_i=P_2(i)$ for even $i>r+1$.
\item We have
$$\lim_{i \to \infty}\frac{\gamma_i}{i^{r-1}}=\frac{1}{2(r-1)!}.$$
Therefore, $P_1,P_2$ are two polynomials with the same leading term $\frac{1}{2(r-1)!}i^{r-1}$. In particular, $\gamma_i=\frac{1}{2(r-1)!}i^{r-1}+O(i^{r-2})$.
\end{enumerate}    
\end{proposition}
\begin{proof}
(1) is straightforward. (2) comes from the fact that $\gamma_{R/I}(z)=(f_1(z^2)+f_2(z^2)z)/(1-z^2)^{r+1}$ where $\deg_z(f_1(z^2)+f_2(z^2)z)=r+1$. (3) is a result by \cite[Lemma 2.2]{arxivLech}.    
\end{proof}
\begin{corollary}\label[corollary]{4.7}
For $n \in \mathbb{Z}$ and $r \geq 2$, $\ell(\Omega^n(k))=\frac{\ell(R)}{2(r-1)!}|n|^{r-1}+O(|n|^{r-2})$. It is a quasi-polynomial of period 2 for $|n|\gg 0$.   
\end{corollary}
\begin{proof}
For $n \geq 0$ this is a consequence of \Cref{4.4} and \Cref{4.6}; for $n<0$ we have $\ell(\Omega^n(k))=\ell(\Omega^{-n}(k^*))=\ell(\Omega^{-n}(k))$.    
\end{proof}
Now we study the direct sum decomposition of tensor products of $\Omega^n(k)$. Recall that the stable module category over a ring $R$ is the category whose objects are the same as module category, and whose homomorphisms are the group of $R$-linear maps modulo those maps factorizing through a projective object. If $M,N$ are two modules, then $M=N$ in the stable module category if and only if $M\oplus F=N\oplus G$ for some free modules $F,G$.
\begin{proposition}\label[proposition]{4.8}
In the stable module category, we have:
\begin{enumerate}
\item $M\otimes_kN,\Omega(M),\Omega^{-1}(M)$ are well-defined.
\item $\Omega^n(M)=M\otimes_k \Omega^n(k)$.
\item $\Omega^n(k)\otimes_k \Omega^m(k)=\Omega^{n+m}(k)$.
\item $\Omega^m(M)\otimes_k\Omega^n(N)=\Omega^{n+m}(M\otimes_kN)$.
\end{enumerate}    
\end{proposition}
\begin{proof}
(1) $M\otimes_kN$ is well-defined since $N$ is projective implies $M\otimes_kN$ is projective for any $M$. If $M$ is projective, $\Omega(M)=\Omega^{-1}(M)=0$ is projective, so $\Omega,\Omega^{-1}$ is well-defined.

(2) There are two exact sequences
$$0 \to M\otimes_k \Omega^n(k) \to M\otimes_kF \to M\otimes_k\Omega^{n-1}(k) \to 0$$
and
$$0 \to \Omega^n(M) \to G \to \Omega^{n-1}(M) \to 0,$$
where $M\otimes_kF,G$ are free. Apply Schanuel's Lemma and induct on $n$.

(3) This is the special case of (2) when $M=\Omega^{m}(k)$.

(4) This is true by (2) and (3) and commutativity of $\otimes$.
\end{proof}
We construct two rings. The first ring is $\Lambda$ whose elements correspond to $R$-modules and multiplication corresponds to the tensor product as follows. The second ring is an overring $\Lambda'$ whose ring structure is simpler.
\begin{definition}
\begin{enumerate}
\item Let $\Lambda$ be the free abelian group with basis $v_n=[\Omega^n(k)],n \in \mathbb{Z}$ and $f=[R]$ where $[\Omega^n(k)]$ and $[R]$ are symbols.
\item For modules of the form $M=\oplus_{n \in \mathbb{Z}}(\Omega^n(k))^{a_n}\oplus R^b$, denote $[M]=\sum_{n \in \mathbb{Z}}a_nv_n+bf$.
\item Define a multiplication structure on $\Lambda$ such that $[M][N]=[M\otimes_kN]$; $\Lambda$ is closed under product by \Cref{4.8}. This makes $\Lambda$ a commutative ring with identity $v_0=[k]$.
\item Let $D=|G|=\dim_kR=p^{e_1+e_2+\ldots+e_r}$, and let $\Lambda' \supseteq \Lambda$ be the free abelian group with basis $v_n,n \in \mathbb{Z}$ and $\frac{1}{D}f$.
\item The dimension function $\dim:\Lambda' \to \mathbb{Z}$ is a $\mathbb{Z}$-linear function sending $v_n$ to $\dim_k \Omega^n(k)$ and sending $f$ to $D$. Set $\bar{v_i}=v_i-\dim(v_i)/D\cdot f \in \Lambda'$.   
\end{enumerate} 
\end{definition}
Now we explore the ring structure of $\Lambda'$. We see for any $\lambda \in \Lambda'$, $\lambda\cdot f=(\dim \lambda)f$. Therefore, $\frac{1}{D}f$ is an idempotent of $\Lambda'$ and we have a ring isomorphism
$$\Lambda' \cong \Lambda_1'\times \Lambda_2'=(v_0-\frac{1}{D}f)\Lambda'\times \frac{1}{D}f\Lambda'$$
which sends $\lambda$ to $((v_0-\frac{1}{D}f)\lambda,\frac{\dim \lambda}{D}f)$. Direct computation shows that $\bar{v_i}=(v_0-\frac{1}{D}f)v_i$, so $\Lambda_1'=\oplus_{i \in \mathbb{Z}}\mathbb{Z}\bar{v_i}$ and $\Lambda_2'=\mathbb{Z}(\frac{1}{D}f)$. \Cref{4.8} says $\bar{v_i}\bar{v_j}=\overline{v_{i+j}}+cf$ for some $c \in \frac{1}{D}f$; by the ring structure the $\Lambda'_2$-component of $\bar{v_i}\bar{v_j}$ must vanish, so $c=0$ and $\bar{v_i}\bar{v_j}=\overline{v_{i+j}}$. Therefore, there is an isomorphism of rings
$$\Lambda' \to \mathbb{Z}[t,t^{-1}]\times \mathbb{Z}$$
sending $\bar{v_i}$ to $t^i$ and $\frac{1}{D}f$ to $1$ in the second copy of $\mathbb{Z}$.
\subsection{Core of a module and the asymptotic behavior of an example}
We recall the following definitions for modular representations of a finite group, which are studied in \cite{BensonBanach} and \cite{BensonGamma}.
\begin{definition}
Let $G$ be a finite group. Assume all representations are over a field $k$.
\begin{enumerate}
\item Let $M$ be a $G$-representation, then there is a decomposition $M=M_1\oplus M_2$ where $M_2$ is projective and $M_1$ has no projective summands. We say $M_1$ is the core of $M$, denoted by $\operatorname{core}(M)$.
\item For $n\geq 1$, we denote $c^G_n(M)=\dim_k \operatorname{core}(M^{\otimes n})$ and $\gamma_G(M)=\lim_{n \to \infty}\sqrt[n]{c^G_n(M)}$.
\end{enumerate}    
When $n=0$, we adopt the convention that $M^{\otimes 0}=k$ and $c^G_n(M)=1$.
\end{definition}
We use the notation $f(n)\sim g(n)$, which stands for $\lim_{n \to \infty}\frac{f(n)}{g(n)}=1$. According to the definition, $\log c^G_n(M) \sim n\log\gamma_G(M)$. The invariant $\gamma_G(M)$ is introduced in \cite{BensonGamma}, where the authors prove its existence and study many of its properties. A next thing to study is the asymptotic behavior of $c^G_n(M)/\gamma_G(M)^n$. We start with two examples. 
\begin{example}\label[example]{4.11}
Let $G=V_4=\mathbb{Z}/2\times\mathbb{Z}/2$. Then the syzygies and cosyzygies of $k$ satisfy $\dim \Omega^n(k)=2|n|+1$. Take $M=\Omega^1(k)\oplus\Omega^{-1}(k)$. Since $\Omega^m(k)\otimes\Omega^n(k)=\Omega^{m+n}(k)\oplus (\textup{projective})$, we see
$$\operatorname{core}(M^{\otimes n})=\oplus_{0 \leq i \leq n} (\Omega^{n-2i}(k))^{{n\choose i}}.$$
Thus
$$c^G_n(M)=\sum_{0 \leq i \leq n}{n\choose i}(|2n-4i|+1)=2^n+2\sum_{0 \leq i \leq n}{n\choose i}|n-2i|.$$

Case 1: $n=2m$. Then $c^G_n(M)=2^n+8\sum_{0 \leq i \leq m-1}{2m\choose i}(m-i)$. We set
$$A=\sum_{0 \leq i \leq m-1}{2m\choose i},B=\sum_{0 \leq i \leq m-1}i{2m\choose i}=\sum_{0 \leq i \leq m-1}2m{2m-1\choose i-1},$$
then
$$A=\frac{1}{2}(2^{2m}-{2m\choose m}),B=2m(\frac{1}{2}2^{2m-1}-{2m-1\choose m-1}),$$
$$c^G_n(M)=2^n+8(mA-B)=2^n+4m{2m\choose m}=2^n+2n{n\choose n/2}.$$

Case 2: $n=2m+1$. Then $c^G_n(M)=2^n+4\sum_{0 \leq i \leq m}{2m+1\choose i}(2m+1-2i)$. We set
$$A=\sum_{0 \leq i \leq m}{2m+1\choose i},B=\sum_{0 \leq i \leq m}i{2m+1\choose i}=\sum_{0 \leq i \leq m-1}(2m+1){2m\choose i-1},$$
then
$$A=2^{2m},B=(2m+1)\frac{2^{2m}-{2m\choose m}}{2},$$
$$c^G_n(M)=2^n+4((2m+1)A-2B)=2^n+4(2m+1){2m\choose m}=2^n+4n{n-1\choose (n-1)/2}.$$

We apply Stirling's formula $n!\sim \sqrt{2\pi n}(\frac{n}{e})^n$ to get
$$2n{n\choose n/2}\sim \sqrt{\frac{8n}{\pi}}2^n\sim4n{n-1\choose (n-1)/2}.$$
This example reveals the following behavior:
\begin{enumerate}
\item $\gamma_G(M)=2$, which is the number of direct summands of $M$;
\item $c^G_n(M)\sim\gamma_G(M)^n\cdot cn^{1/2}$ for some number $c$.
\end{enumerate}
\end{example}
\begin{example}\label[example]{4.12}
Let $G=V_4$, $M=k\oplus \Omega^1(k)$. Then  
$$\operatorname{core}(M^{\otimes n})=\oplus_{0 \leq i \leq n} (\Omega^{i}(k))^{{n\choose i}},$$
and
$$c^G_n(M)=\sum_{0 \leq i \leq n}{n\choose i}(2i+1)=2^n+2n\sum_{0 \leq i \leq n}{n-1\choose i-1}=(n+1)2^n\sim n2^n.$$
In this example $\gamma_G(M)=2$ and $c^G_n(M)\sim\gamma_G(M)^n\cdot cn$ for some $c$.
\end{example}
The above two examples indicate that in certain examples, $c^G_n(M)/\gamma_G(M)^n$ may asymptotically look like $cn^\alpha$ where $\alpha \in \frac{1}{2}\mathbb{Z}$. In the next subsection, we will prove that this phenomenon happens to direct sums of syzygies and cosyzygies of $k$ and express $c,\alpha$ in terms of data of the syzygy or cosyzygy and data of the underlying group.
\subsection{Limit behavior of core of tensor power of syzygies}
In this subsection, we will apply the language of probability theory to derive some asymptotic behavior of $c^G_n(M)$. In general, a random variable $X$ is a measurable function over certain measure space; in our application the space is $\mathbb{R}^n$ for some $n$ and the measure is a Borel measure. For other notations, such as the expectation, mean, variance and probability density function, we refer to standard textbooks in probability theory like \cite{Durrett}.

We denote $\Lambda_0=\Lambda_1\otimes_\mathbb{Z}\mathbb{R}=\oplus_{i \in \mathbb{Z}} \mathbb{R}v_i$, which is an $\mathbb{R}$-vector space over a basis indexed by $\mathbb{Z}$. $\Lambda_0$ is an $\mathbb{R}$-algebra whose multiplication is given by $v_iv_j=v_{i+j}$ and $\mathbb{R}$-linearity. The subset of $\Lambda_0$ of elements with nonnegative components is $\Lambda_0^+=\oplus_{i \in \mathbb{Z}} \mathbb{R}_{\geq 0}v_i$. For $\lambda=\lambda_iv_i \in \Lambda_0$, denote $\|\lambda\|=\sum_{i \in \mathbb{Z}} |\lambda_i|$. We say $\lambda$ is normalized if $\|\lambda\|=1$. In general, we can normalize a nonzero $\lambda$: $\lambda=\|\lambda\|\cdot \lambda^{nor}$ such that $\|\lambda^{nor}\|=1$. If $\|\lambda\|=1$ and $\lambda \in \Lambda_0^+$, then $\lambda$ corresponds to a discrete probablity measure. We denote the corresponding random variable by $X_\lambda$, where $P(X_\lambda=i)=\lambda_i$.

\begin{lemma}
Let $\lambda_1,\lambda_2\in \Lambda_0^+$ be two normalized elements. Then $X_{\lambda_1\lambda_2}=X_{\lambda_1}+X_{\lambda_2}$ where $X_{\lambda_1},X_{\lambda_2}$ are independent random variables. Consequently, if $\lambda \in \Lambda_0^+$ corresponds to a random variable $X$, then $X_{\lambda^n}=Y_n=X_1+\ldots+X_n$ which is the sum of $n$ independent variable with same probability distribution as $X$.
\end{lemma}
\begin{proof}
We write $\lambda_1=\sum_i \lambda_{1,i}v_i$ and $\lambda_2=\sum_i \lambda_{2,i}v_i$, then $\lambda_1\lambda_2=\sum_i (\sum_j \lambda_{1,j}\lambda_{2,i-j})v_i$. So
\begin{align*}
P(X_{\lambda_1}+X_{\lambda_2}=i)=\sum_jP(X_{\lambda_1}=j)P(X_{\lambda_2}=i-j|X_{\lambda_1}=j)\\
=\sum_jP(X_{\lambda_1}=j)P(X_{\lambda_2}=i-j)=\sum_j \lambda_{1,j}\lambda_{2,i-j}=P(X_{\lambda_1\lambda_2}=i).    
\end{align*}
Here the second equality uses independence of $X_{\lambda_1},X_{\lambda_2}$.  
\end{proof}

Now assume $\lambda \in \Lambda_0^+$ is normalized, $X_\lambda$ has mean $\mu$ and variance $\sigma^2$. We set $X_{\lambda^n}=Y_n=X_1+\ldots+X_n$ where $X_i$'s are independent and have the same distribution as $X_\lambda$. The central limit theorem says:
$$Z_n=\frac{Y_n-n\mu}{\sigma\sqrt{n}}\to Z \sim N(0,1).$$
In other words, for any $\alpha<\beta$, the sum of coefficients of $v_i$ in $\lambda^n$ for $i \in [n\mu+\sigma\sqrt{n}\alpha,n\mu+\sigma\sqrt{n}\beta]$ is approximately $\int_\alpha^\beta f(x)dx$ where $f(x)=\frac{1}{\sqrt{2\pi}}e^{-\frac{x^2}{2}}$ is the probability density function of $N(0,1)$. Let $Z\sim N(0,1)$ be a random variable. If $\alpha\geq 0$, we have 
$$E(|Z|^\alpha)=\frac{2}{\sqrt{2\pi}}\int_0^\infty x^\alpha e^{-x^2/2}dx=\frac{2\cdot 2^{(\alpha-1)/2}}{\sqrt{2\pi}}\int_0^\infty u^{(\alpha-1)/2}e^{-u}du=2^{\alpha/2}\Gamma(\frac{\alpha+1}{2})/\sqrt{\pi}.$$

Now let $L$ be an $\mathbb{R}$-linear function $\Lambda_0 \to \mathbb{R}$. Let $F:\mathbb{Z} \to \mathbb{R},F(i)=\ell(v_i)$. Then $\ell(\lambda^n)=E(F(Y_n))$.
\begin{lemma}
Let $L:\Lambda_0 \to \mathbb{R}$ be an 
$\mathbb{R}$-linear map, $\lambda, X_\lambda, Y_n,\mu,\sigma$ are as above, $F:\mathbb{Z} \to \mathbb{R}$ be a map such that $\ell(v_i)=F(i)$, $Z \sim N(0,1)$. Assume $F(i)=C|i|^\alpha+O(|i|^{\alpha-1})$ where $\alpha \geq 1$ is a real number and $C \neq 0$ is a constant.
\begin{enumerate}
\item If $\mu\neq 0$, then
$$\lim_{n \to \infty}\frac{E(F(Y_n))}{n^{\alpha}}=C|\mu|^\alpha.$$
In other words, $\ell(\lambda^n)\sim C|\mu|^\alpha n^{\alpha}$.
\item If $\mu=0$, then
$$\lim_{n \to \infty}\frac{E(F(Y_n))}{n^{\alpha/2}}=C\sigma^{\alpha}E(|Z|^\alpha).$$
In other words, if $\sigma \neq 0$, $\ell(\lambda^n)\sim n^{\alpha/2}\cdot C\sigma^{\alpha}E(|Z|^\alpha)$.
\end{enumerate}
\end{lemma}
\begin{proof}
If $\mu \neq 0$, then $Y_n=\sigma\sqrt{n}Z_n+n\mu=n\mu(1+\frac{\sigma Z_n}{\mu\sqrt{n}})$. Therefore, $F(Y_n)=C|Y_n|^\alpha+O(|Y_n|^{\alpha-1})=C|n\mu|^\alpha+O((1+|Z_n|^\alpha)n^{\alpha-1/2})$. Since $E(Z_n^\alpha) \to E(Z^\alpha)<\infty$ for any $\alpha\in \mathbb{R}$, $E(F(Y_n))=C|n\mu|^\alpha+O(n^{\alpha-1/2})$. 

If $\mu=0$, then $Y_n=\sigma\sqrt{n}Z_n$,  $F(Y_n)=C|\sigma\sqrt{n}Z_n|^\alpha+O(|\sqrt{n}Z_n|^{\alpha-1})$, $E(F(Y_n))=n^{\alpha/2}C\sigma^{\alpha}E(|Z_n|^\alpha)+O(n^{(\alpha-1)/2})$ and $Z_n \to Z$. So we are done.
\end{proof}
Now we will state our main result in this section. First we set up the notations for this result.
\begin{enumerate}
\item Let $k$ be a field of characteristic $p>0$, $R=kG$ where $\dim_k R=|G|=p^{e_1+\ldots+e_r}=D$.
\item Let $M=\oplus_{i \in \mathbb{Z}}\Omega^i(k)^{a_i}\neq 0$, $c^G_n(M)=\dim_k \operatorname{core}(M^{\otimes n})$.
\item Let $\lambda_1=\sum_{i \in \mathbb{Z}}a_iv_i \in \Lambda_0^+$,  $\gamma=\|\lambda_1\|=\sum_{i \in \mathbb{Z}}a_i$, $\lambda_2=\frac{1}{\gamma}\lambda_1=\sum_{i \in \mathbb{Z}}\frac{a_i}{\gamma}v_i$.
\item Let $X$ be a random variable associated with $\lambda_2$, that is, $P(X=i)=a_i/\gamma$. The mean and variance of $X$ are $\mu$ and $\sigma^2$ respectively.
\item Let $Z\sim N(0,1)$ be a random variable.
\end{enumerate}
\begin{theorem}\label{4.15}
Under the above notations, if $\mu \neq 0$, then
$$c^G_n(M)\sim \gamma^n\cdot n^{r-1}\cdot |\mu|^{r-1}\cdot \frac{D}{2(r-1)!}.$$
If $\mu=0$ and $\sigma \neq 0$, then
$$c^G_n(M)\sim \gamma^n\cdot n^{(r-1)/2}\cdot \sigma^{r-1}\cdot E(|Z|^{r-1})\cdot \frac{D}{2(r-1)!}.$$
If $\mu=\sigma=0$, then
$$c^G_n(M)=\gamma^n.$$
\end{theorem}
\begin{proof}
We define $L:\Lambda_0 \to \mathbb{R},F:\mathbb{Z} \to \mathbb{R}$ such that $F(i)=\ell(v_i)=\dim_k \Omega^i(k)$, then by \Cref{4.7},
$$F(i)=\ell(v_i)=\frac{D}{2(r-1)!}|i|^{r-1}+O(|i|^{r-2}).$$
Then 
$$c^G_n(M)=\ell(\lambda_1^n)=\ell((\gamma \lambda_2)^n)=\gamma^n\ell(\lambda_2^n).$$ 
Now $\lambda_2^n \in \Lambda_0^+$ is normalized which corresponds to random variable $Y_n=X_1+\ldots+X_n$ where $X_i$'s are independent and follow the same probability distribution as $X_{\lambda_2}$, so
$$\ell(\lambda_2^n)=E(F(Y_n)).$$
In the first case $\mu \neq 0$, so
$$E(F(Y_n))\sim\frac{D}{2(r-1)!}n^{r-1}|\mu|^{r-1}.$$
In the second case $\mu=0$ and $\sigma \neq 0$, so
$$E(F(Y_n))\sim\frac{D}{2(r-1)!}n^{(r-1)/2}\sigma^{r-1}E(|Z|^{r-1}).$$
Combining the above inequalities, we get the result in these two cases. Finally, if $\mu=\sigma=0$, then $M=k^\gamma$, so the last case is trivial.
\end{proof}
\begin{example}
We use the above theorem to verify \Cref{4.11} and \Cref{4.12}. In the setting of \Cref{4.11}, we have $r=2,e_1=e_2=1, D=2^{e_1+e_2}=4, \gamma=2$. The probability distribution of $X$ is $P(X=\pm 1)=1/2$, so its mean $\mu=0$ and variance $\sigma^2=1$. We see if $Z\sim N(0,1)$,
$$E(|Z|)=\int_\mathbb{R}|z|\frac{1}{\sqrt{2\pi}}e^{-z^2/2}dz=\sqrt{\frac{2}{\pi}}.$$
Therefore,
$$c^G_n(M)\sim 2^n\cdot n^{(2-1)/2}\cdot \sqrt{\frac{2}{\pi}}\cdot\frac{4}{2\cdot 1!}=\sqrt{\frac{8n}{\pi}}2^n.$$
which agrees with direct computation in \Cref{4.11}.

In the setting of \Cref{4.12}, $r,e_1,e_2,D,\gamma$ are the same, $P(X=0)=P(X=1)=1$, so $\mu=1/2,\sigma^2=1/4$. We see
$$c^G_n(M)\sim 2^n\cdot n^{2-1}\cdot \frac{1}{2}\cdot\frac{4}{2\cdot 1!}=n2^n.$$
\end{example}
\begin{remark}
From the result we see that when $k,G$ are fixed, the asymptotic behavior of $c^G_n(M)$ only depends on the number of direct summands and the mean and the variance of the associated probability distribution, and is independent of the concrete distribution.    
\end{remark}
The same method can be applied to any additive functor on $\operatorname{core}(M^{\otimes n})$ with good asymptotic behavior. For example, in \cite{CHU}, the authors have considered the following sequence
$$d^G_n(M)=\dim_k \operatorname{soc}\operatorname{core}(M^{\otimes n}).$$
We now derive its asymptotic behavior.
\begin{theorem}\label{4.18}
Under the same notations as \Cref{4.15}, if $\mu \neq 0$, then
$$d^G_n(M)\sim \gamma^n\cdot n^{r-1}\cdot |\mu|^{r-1}\cdot \frac{1}{(r-1)!}.$$
If $\mu=0$ and $\sigma \neq 0$, then
$$d^G_n(M)\sim \gamma^n\cdot n^{(r-1)/2}\cdot \sigma^{r-1}\cdot E(|Z|^{r-1})\cdot \frac{1}{(r-1)!}.$$
If $\mu=\sigma=0$, then
$$d^G_n(M)=\gamma^n.$$
\end{theorem}
\begin{proof}
We see
$$\dim_k\operatorname{soc}(\Omega^n(k))=r(\Omega^n(k))=
\begin{cases}
\beta_{n-1}(k) & n\geq 1\\
\beta_{|n|}(k) & n \leq 0.
\end{cases}
$$
Therefore, from Tate's resolution we see
$$r(\Omega^i(k))=\frac{|i|^{r-1}}{(r-1)!}+O(|i|^{r-2})$$
and proceed as in \Cref{4.15}.
\end{proof}
\subsection{$\Omega$-algebraic modules, and a question by Benson and Symonds}
We recall the concepts of $\Omega$-algebraic and $\Omega^{\pm}$-algebraic modules introduced in \cite{BensonGamma} and \cite{CHU}.
\begin{definition}[\cite{CHU}]
A $kG$-module $M$ is called $\Omega$-algebraic if the non-projective indecomposable direct summands of the modules $M^{\otimes n}$
fall into finitely many orbits of both
functors $\Omega$ and $\Omega^{-1}$. If we replace $\Omega$ and $\Omega^{-1}$ by just $\Omega$, we say it is $\Omega^+$-algebraic, and if we replace $\Omega$ and $\Omega^{-1}$ by just $\Omega^{-1}$, we say it is $\Omega^-$-algebraic.  
\end{definition}
We say a function $\phi:\mathbb{N} \to \mathbb{N}$ is eventually recursive if there is $d,a_1,\ldots,a_d$ such that $\phi(n)+a_1\phi(n-1)+\ldots+a_n\phi(n-d)=0$ for $n\gg 0$. In \cite{BensonGamma}, Benson and Symonds ask the following question:
\begin{question}[\cite{BensonGamma}, Question 14.2]
If $M$ is $\Omega$-algebraic, is $c^G_n(M)$ eventually recursive in terms of $n$?   
\end{question}
This is a weaker version of \cite[Question 13.3]{BensonGamma} asking whether $c^G_n(M)$ is eventually recursive for any module $M$. The main results in \cite{CHU} are as follows.
\begin{enumerate}
\item If $M$ is $\Omega$-algebraic, then $\sum_{n \geq 0}c^G_n(M)t^n$ is algebraic.
\item If $M$ is $\Omega^+$ or $\Omega^-$-algebraic, then $\sum_{n \geq 0}c^G_n(M)t^n$ is rational and $c^G_n(M)$ is eventually recursive.
\end{enumerate}
It is easy to see $\Omega^n(k)$ for any $n$ and their direct sums are $\Omega$-algebraic, since any tensor power of such modules lies in the orbit of $k$. They are the ``simpliest" kind of $\Omega$-algebraic modules. Now we prove that even among direct sums of such kind of modules, there are examples where $c^G_n$ is not eventually recursive.
\begin{theorem}\label{4.21}
Let $G$ be an abelian $p$-group minimally generated by $r$ elements where $r$ is even. Then for $M=\oplus_{i \in \mathbb{Z}}\Omega^i(k)^{a_i}\neq 0$ with $\sum ia_i=0$, $M$ is $\Omega$-algebraic while $c^G_n(M)$ is not eventually recursive.
\end{theorem}
\begin{proof}
We define the random variable $X$ associated with $M$ as in the notations for \Cref{4.15}, then $\mu=0$ and $\sigma \neq 0$. \Cref{4.15} says
$$c^G_n(M)\sim C\gamma^nn^{(r-1)/2}$$
where $C\neq 0$ and $(r-1)/2 \notin \mathbb{Z}$. However, for any eventually recursive function $\phi$ we have that there is a positive integer $d$ such that in each residue class modulo $d$, $\phi(n)\sim C_n\gamma^nn^k$ where $k$ is an integer and $C_n$ is either a constant or an oscillating bounded function. So $c^G_n(M)$ cannot be an eventually recursive function.
\end{proof}
For example, the module $M=\Omega(k)\oplus \Omega^{-1}(k)$ over $V_4$ appearing in \Cref{4.11} satisfies the condition of \Cref{4.21}.

\section{Computation of core series of $\Omega$-algebraic modules}
\subsection{Core series}
We begin with the following definition.
\begin{definition}
Let $G$ be a finite group and $M$ be a modular representation of $G$. \textbf{The core series} of $M$ is    
$$c_M(z)=\sum_{n \geq 0}c^G_n(M)z^n.$$
\end{definition}
This series will be the main object to study in the following two sections.

The core series is mainly studied in \cite{BensonGamma} and \cite{CHU}. In \cite{BensonGamma}, the invariant $\gamma_G(M)$ is just the radius of convergence of $c_M(z)$; the result in \cite{CHU} says the core series of $\Omega$-algebraic modules are algebraic, and the core series of $\Omega^+$ or $\Omega^-$-algebraic modules are rational. Actually, a series is rational if and only if it is a generating function of an eventually recursive sequence. So, it is reasonable to study the whole sequence as a power series instead of a sequence indexed by integers. \Cref{4.21} in last section just says some $\Omega$-algebraic modules have core series that are algebraic but not rational.

Now the following question is natural.
\begin{question}\label[question]{5.2}
What kind of properties does a core series have, if it is algebraic but not rational?    
\end{question}
The following result is due to Benson in a personal communication, which is a reformulation of \Cref{4.11}. 
\begin{example}\label[example]{5.3}
Let $G=V_4$ and $M=\Omega(k)\oplus \Omega^{-1}(k)$, then
$$c_M(z)=(1+2z)(2/(1-4z^2)+4/(1-4z^2)^{3/2}).$$
\end{example}
In this section, we will present a systematic way to compute the core series of $\Omega$-algebraic modules. One direction for an answer to \Cref{5.2} is the upper bound of the degree of an algebraic series over $\mathbb{C}(z)$, and we will give such an upper bound depending on the data coming from the decomposition of tensor products. This relies on an algebraic version of Cauchy's residue theorem, which will be introduced in the next subsection.

\subsection{Formal Cauchy's residue theorem}
The well-known Cauchy's residue theorem allows us to compute contour integral in terms of residue at the poles of a meromorphic function. In this subsection, we will introduce the formal Cauchy's residue theorem, which works over the field of formal Laurent series over any field that does not necessarily have an analytic structure.

The formal Cauchy's residue theorem has been widely used in analytic combinatorics to evaluate the generating function of some combinatorial series. For example, see \cite{Cauchy1}\cite{Cauchy2}\cite{Cauchy3}. The Lagrange inversion theorem for inverse function is a direct application of this theorem in characteristic $0$. However, most references do not give the full statement of formal Cauchy's residue theorem. Therefore, we present a detailed proof of one version we need in this subsection. 

In this subsection, we consider the following table of monomials in two variables $u,z$:
\begin{center}
\begin{tabular}{ccccc}
 & $\vdots$ & $\vdots$ & $\vdots$ &  \\
$\ldots$ & $uz^{-1}$ & $u$ & $uz$ & $\ldots$ \\
$\ldots$ & $z^{-1}$ & 1 & $z$ & $\ldots$ \\
$\ldots$ & $u^{-1}z^{-1}$ & $u^{-1}$ & $u^{-1}z$ & $\ldots$ \\
 & $\vdots$ & $\vdots$ & $\vdots$ &  \\
\end{tabular}    
\end{center}
Linear combinations of these terms over $\mathbb{C}$ which are possibly an infinite sum form different rings. For example, $\mathbb{C}((u))((z))$ is the linear combination such that in each column the nonzero terms are bounded from below, and all the nonzero terms are bounded from left uniformly in rows. Every nonzero polynomial is invertible in this ring, so we can realize $\mathbb{C}(u)$ and $\mathbb{C}(z)$ as subrings of $\mathbb{C}((u))((z))$.

We can also consider the Puiseux series with bounded denominator $\mathbb{C}((z^{1/d}))$. It is well-known that the algebraic closure of the field of Laurent power series is the field of Puiseux series; in particular, if there are finitely many algebraic elements over $\mathbb{C}((z))$, we may assume they lie in $\mathbb{C}((z^{1/d}))$ for one $d$. We will be considering the ring $\mathbb{C}((u))((z^{1/d}))$ and its subrings in the following part, which can be viewed as the rings of certain linear combinations of monomials in $u,z^{1/d}$.

\begin{definition}
Let $f(u,z)=\sum_{m,n}a_{m,n}u^mz^n \in \mathbb{C}((u))((z))$. Define
$$[u^i]f(u,z)=\sum_{n}a_{i,n}z^n \in \mathbb{C}((z)).$$
\end{definition}
The above definition is the analogue of contour integral.
\begin{definition}
Let $h(u,z) \in \mathbb{C}((z))[u]$. It can be viewed as a polynomial in $u$, and let $\rho$ be its root. We say $\rho$ is a small root if $\rho=0$ or $\operatorname{ord}_z\rho(z)>0$, otherwise we say it is a big root. Let $f(u,z) \in \mathbb{C}((u))((z))$. Suppose $f(u,z)=g(u,z)/h(u,z)$ such that $g(u,z) \in \mathbb{C}[[u,z]]$ and $h(u,z) \in \mathbb{C}((z))[u]$ with no common factor, or in other words, $g(\rho(z),z)\neq 0$ for any small root $\rho(z)$ of $h$, then we say a small root $\rho(z)$ of $h$ is a pole of $f$.  
\end{definition}
\begin{remark}
Let $\rho(z)$ be a big root with valuation $-\nu,\nu \in \mathbb{Q}_{\geq 0}$. Then $\rho(z)=z^{-\nu}\phi(z)$ with $\phi(z)$ invertible in $\mathbb{C}[[z^{1/d}]]$. Therefore,
$$1/(u-\rho(z))=1/(u-z^{-\nu}\phi(z))=-z^{\nu}\phi^{-1}(z)/(1-uz^\nu\phi^{-1}(z))=-z^{\nu}\phi^{-1}(z)\sum_{i \geq 0}(uz^\nu\phi^{-1}(z))^i$$
is an element in $\mathbb{C}[[u,z^{1/d}]]$. Therefore we may always assume $h$ has no big roots, so when we consider the poles, we only need to look at the small roots.
\end{remark}
\begin{proposition}[Partial fraction decomposition]
Let $f(u,z) \in \mathbb{C}((u))((z))$. Suppose $f(u,z)=g(u,z)/h(u,z)$ such that $g(u,z) \in \mathbb{C}[[u,z]]$ and $h(u,z) \in \mathbb{C}((z))[u]$. Let $u=\rho_1(z),\ldots,\rho_r(z)$ be all the poles of $f(u,z)$, which has multiplicity $d_1,\ldots,d_r$ respectively. Then we can write
$$f(u,z)=f_0(u,z)+\sum_{1 \leq i \leq r,1 \leq j \leq d_i} f_{ij}(z)/(u-\rho_i(z))^j$$
where $f_0(u,z) \in \mathbb{C}[[u]]((z^{1/d}))$ has no poles and $f_{ij}(z) \in \mathbb{C}((z^{1/d}))$. Moreover if $f(u,z) \in \mathbb{C}(u,z)$, then $f_{ij}(z) \in \mathbb{C}(z,\rho_i(z))$ and $f_0(u,z) \in \mathbb{C}(u,\rho_1(z),\ldots,\rho_r(z))$.
\end{proposition}
\begin{proof}
Using the partial fraction decomposition over $\mathbb{C}((z^{1/d}))$ for some $m$, we can write    
$$f(u,z)=f_0(u,z)+\sum_{1 \leq j \leq r} g_i(u,z)/(u-\rho_i(z))^{d_i}.$$
For $g_i(u,z) \in \mathbb{C}[[u]]((z^{1/d}))$. But we can always write $g_i(u,z)=\sum_{0\leq i \leq r-1}(u-\rho_i(z))^if_{r-i}(z)+h_i(u,z)$, so the equation holds. Moreover, for each fixed $i$, we can compute $f_{ij}(z)$ inductively: $f_{i,d_i}(z)=f(u,z)(u-\rho_i(z))^{d_i}|_{u=\rho_j(z)}$, $f_{i,d_{i-1}}=f(u,z)(u-\rho_i(z))^{d_{i-1}}-f_{d_i}(z)(u-\rho_i(z))^{-1}|_{u=\rho_i(z)},\ldots$. From these equations we deduce that if $f(u,z) \in \mathbb{C}(u,z)$ then $f_{ij}(z) \in \mathbb{C}(z,\rho_i(z))$ and $f_0(u,z) \in \mathbb{C}(u,\rho_1(z),\ldots,\rho_r(z))$.
\end{proof}
\begin{definition}
In above proof, define the residue of $f$ at $u=\rho_i(z)$ to be
$$\operatorname{Res}_{u=\rho_i(z)}f(u,z)=f_{i,-1}(z).$$
\end{definition}
\begin{remark}
The residue is easy to compute for simple poles: when $u=\rho(z)$ is a simple pole, we have
$$\operatorname{Res}_{u=\rho(z)}f(u,z)=f(u,z)(u-\rho(z))|_{u=\rho(z)}.$$
Moreover if $f(u,z)=g(u,z)/h(u,z)$ such that $g(u,z) \in \mathbb{C}[[u,z]]$ and $h(u,z) \in \mathbb{C}((z))[u]$, and $\rho(z)$ is a simple root of $h(u,z)$, then
$$\operatorname{Res}_{u=\rho(z)}f(u,z)=\frac{g(u,z)}{h'_u(u,z)}|_{u=\rho(z)}.$$
\end{remark}
\begin{theorem}[Formal Cauchy's residue theorem]\label{5.10}
Let $f(u,z) \in \mathbb{C}((u))((z))$. Suppose $f(u,z)=g(u,z)/h(u,z)$ such that $g(u,z) \in \mathbb{C}[[u,z]]$ and $h(u,z) \in \mathbb{C}((z))[u]$. Let $u=\rho_1(z),\ldots,\rho_r(z)$ be all the poles of $f(u,z)$, then
$$[u^{-1}]f(u,z)=\sum_{1 \leq i \leq r}\operatorname{Res}_{u=\rho_i(z)}f(u,z).$$
\end{theorem}
\begin{proof}
We calculate the expansion of $(u-\rho(z))^{-i}$ in $\mathbb{C}((u))((z))$. We have
$$\frac{1}{(u-\rho(z))^i}=\frac{1}{u^i}\cdot \frac{1}{(1-u^{-1}\rho(z))^i}=\frac{1}{u^i}\sum_{j \geq 0}{n+i-1\choose i-1}(u^{-1}\rho(z))^i.$$
This means in the partial fraction decomposition, only the terms $1/(u-\rho_i(z))^{-1}$ can contribute one $1/u$ in $[u^{-1}]f(u,z)$. Collecting all their coefficients which are exactly the residues, we get the theorem.
\end{proof}
\subsection{Core series of $\Omega$-algebraic modules}
We will show that formal Cauchy's residue theorem is essential in extracting the behavior of core series. For $x \in \mathbb{R}$, denote $[x]^+=\max\{x,0\},[x]^-=-\min\{x,0\}$. We have $|x|=2[x]^+-x=2[x]^-+x$.
\begin{lemma}
Let $f(z,u)=\sum_{m,n}a_{m,n}u^mz^n$. Then
\begin{enumerate}
\item $\sum_{m,n}ma_{m,n}u^mz^n=u\frac{\partial f}{\partial u}(u,z)$.
\item $\sum_n(\sum_{m}ma_{m,n})z^n=\frac{\partial f}{\partial u}(1,z)$.
\item $\sum_n(\sum_{m}[m]^-a_{m,n})z^n=[u^0]f(u,z)u/(1-u)^2$.
\item $\sum_n(\sum_{m}|m|a_{m,n})z^n=[u^0]2f(u,z)u/(1-u)^2+\frac{\partial f}{\partial u}(1,z)$.
\end{enumerate}
\end{lemma}
\begin{proof}
Here (1) and (2) are obvious and (4) is a consequence of (3) and the equation $|x|=2[x]^-+x$. So we prove (3). We have $u/(1-u)^2=\sum_{i \geq 0}iu^i$, so the $u^0$-coefficient of $\sum_{m,n}a_{m,n}u^mz^n\cdot\sum_{i \geq 0}iu^i$ is exactly $\sum_{i \leq 0}\sum_n-ia_{i,n}z^n$, so we get the result.
\end{proof}
\begin{lemma}\label[lemma]{5.12}
Let $s \geq 1$ be a positive integer.
\begin{enumerate}
\item $\sum_{m,n}m(m-1)\ldots(m-s+1)a_{m,n}u^mz^n=u^s\frac{\partial^s f}{\partial u^s}(u,z)$.
\item $\sum_n(\sum_{m}[m]^-([m]^-+1)\ldots([m]^-+s-1)a_{m,n})z^n=(-1)^s[u^0]\frac{\partial^s f}{\partial u^s}(u,z)u/(1-u)$.
\end{enumerate}    
\end{lemma}
\begin{proof}
(1) is obvious, so we only need to prove (2). We see $u/(1-u)=\sum_{i \geq 1}u^i$, so the right side is equal to the $u^0$-coefficient of  
$$(-1)^s\sum_{m,n}m(m-1)\ldots(m-s+1)a_{m,n}u^mz^n\cdot\sum_{i \geq 1}u^i,$$
which is equal to
$$(-1)^s\sum_{i\geq 1,n}-i(-i-1)\ldots(-i-s+1)a_{-i,n}z^n=(-1)^s\sum_{i\leq -1,n}i(i-1)\ldots(i-s+1)a_{i,n}z^n.$$
It suffices to verify
$$[m]^-([m]^-+1)\ldots([m]^-+s-1)=\begin{cases}
(-1)^sm(m-1)\ldots(m-s+1) & m \leq -1\\
0 & m \geq 0,
\end{cases}$$
and it holds, so (2) is true.
\end{proof}
\begin{lemma}\label[lemma]{5.13}
Let $P$ be any bivariate polynomial, $f(u,z)=\sum_{m,n}a_{m,n}u^mz^n \in \mathbb{C}((u))((z))$. Then
$$\sum_{m,n}P(m,[m]^-)a_{m,n}z^n$$
can be expressed from $f(u,z)$ using a linear combination of the composition of the following functionals
$$|_{u=1},\cdot g(u), g(u) \in \mathbb{C}[u,u^{-1},(1-u)^{-1}],\frac{\partial}{\partial u},[u^0].$$
In particular, for any polynomial $P$,
$$\sum_{m,n}P([m]^+,[m]^-)a_{m,n}z^n$$
can be expressed using these functionals.
\end{lemma}
\begin{lemma}\label[lemma]{5.14}
Let $f(u,z)=\sum_{m,n}a_{m,n}u^mz^n$. Let $Q \in \mathbb{N}$ be a positive integer and $0 \leq s <Q$ be an integer. Let $\xi_Q$ be the $Q$-th root of unity. Denote
$$[u^{Q\mathbb{Z}+s}]f(u,z)=\sum_{m\equiv s \operatorname{mod}Q,n}a_{m,n}u^mz^n.$$
Then
$$[u^{Q\mathbb{Z}+s}]f(u,z)=\frac{1}{Q}\sum_{0 \leq i \leq Q-1}\xi^{-is}_Qf(\xi^i_Qu,z).$$
\end{lemma}
Let $M$ be an $\Omega$-algebraic module such that $\operatorname{core}M^{\otimes n}$ lies in the $\Omega^{\pm}$-orbits of $M_1,\ldots,M_r$ for any $n \geq 1$. Let $V$ be a free $\mathbb{C}[u,u^{-1}]$-module with basis $v_1,\ldots,v_r$ which are row vectors. Let $A=(a_{ij}) \in \operatorname{Mat}_{r}(\mathbb{C}[u,u^{-1}])$ be the following matrix: if in the stable category
$$M_i\otimes M=\oplus_{j,m}\Omega^m(M_j)^{\oplus c_{ijm}},$$
then
$$a_{ij}=\sum_{m}c_{ijm}u^m.$$
In this case, the representation group of $\Omega$-twists of $M_1,\ldots,M_r$ corresponds to $\mathbb{Z}[u,u^{-1}]$-module with basis $v_1,\ldots,v_r$ under
$$[\Omega^i(M_j)] \to u^iv_j,$$
and the action of $\otimes [M]$ becomes right multiplication by $A$ on the row vectors.

We choose a module $N$ such that $v=[N] \in V$. For any $n \geq 0$, $[N\otimes M^{\otimes n}]=vA^n=(vA^nv_1^T,\ldots,vA^nv_r^T)$. Fix $v$ here, we write
$$vA^nv_j^T=\sum_{m}c_{n,j,m}u^m,$$
then
$$N\otimes M^n=\oplus_{j,m} \Omega^m(M_j)^{c_{njm}}.$$
We see for each $j$, $\ell(\Omega^m(M_j))$ is a quasi-polynomial. We fix a common period $Q$ for these polynomials. Then $\ell(\Omega^m(M_j))=P_{js}([m]^+,[m]^-)$ for some polynomial $P_{js}$, $|m|\gg 0$ and $m\equiv s \operatorname{mod} Q$. In other words.
$$\ell(\Omega^m(M_j))=P_{js}([m]^+,[m]^-)\mathbf{1}_{m\equiv s \operatorname{mod}Q}+\sum_{i \in S}R_{ij}\bm\delta_i(m),$$
where $\bm\delta_i(m)$ is the Kronecker symbol and $S$ is a finite set. So
\begin{align*}
\ell(\oplus_m \Omega^m(M_j)^{c_{njm}})=\sum_m c_{njm}P_{js}([m]^+,[m]^-)\mathbf{1}_{m\equiv s \operatorname{mod}Q}+\sum_{i,m}R_{ij}\bm\delta_{i}(m)c_{njm}\\
=\sum_m c_{njm}P_{js}([m]^+,[m]^-)\mathbf{1}_{m\equiv s \operatorname{mod}Q}+\sum_{i}R_{ij}c_{nji}.
\end{align*}
Therefore,
$$\ell(N\otimes M^n)=\sum_{j,m} c_{njm}P_{js}([m]^+,[m]^-)\mathbf{1}_{m\equiv s \operatorname{mod}Q}+\sum_{i,j}R_{ij}c_{nji}$$
and
\begin{equation}\label{equation 5.1}
\sum_{n \geq 0}\ell(N\otimes M^n)z^n=\sum_{n\geq 0,m \in \mathbb{Z},j}c_{njm}P_{js}([m]^+,[m]^-)\mathbf{1}_{m\equiv s \operatorname{mod}Q}z^n+\sum_{n \geq 0,i,j}R_{ij}c_{nji}z^n  
\end{equation}
such that
$$f_j(u,z)=\sum_{n \geq 0}vA^nv_j^T\cdot z^n=\sum_{n\geq 0,m \in \mathbb{Z}}c_{n,j,m}u^mz^n.$$
In particular,
$$f_j(u,z)=v(1-zA)^{-1}v_j^T=\frac{1}{\det(1-zA)}v(1-zA)^*v_j^T,$$
where $(1-zA)^*$ is the adjoint of $(1-zA)$.
\begin{theorem}
For any module $N$ which is a direct sum of $\Omega$-twists of $M_1,\ldots,M_r$, The series
$$f(z)=\sum_{n \geq 0}\ell(N\otimes M^n)z^n$$
is a linear combination of the composition of the following types of functionals
$$|_{u=1},\cdot g(u), g(u) \in \mathbb{C}[u,u^{-1},(1-u)^{-1}],\frac{\partial}{\partial u},[u^0]$$
applied to the bivariate functions
$$[u^{Q\mathbb{Z}+s}]f_j(u,z)=\frac{1}{Q}\sum_{0 \leq i \leq Q-1}\xi^{-is}_Qf_j(\xi^i_Qu,z),$$
where
$$f_j(u,z)=v(1-zA)^{-1}v_j^T=\frac{1}{\det(1-zA)}v(1-zA)^*v_j^T.$$
\end{theorem}
\begin{proof}
$f(z)$ is given by the sum in \Cref{equation 5.1}. The first part of the sum satisfies the condition by \Cref{5.12}, \Cref{5.13} and \Cref{5.14}, and the second part of the sum $\sum_{n \geq 0,i,j}R_{ij}c_{nji}z^n$ is a linear combination of $[u^i]f_j(u,z)=[u^0]u^{-i}f_j(u,z)$.    
\end{proof}
Here we remark that $A$ is a matrix of Laurent series in $u$; so $\det(1-zA)$ is also a Laurent series in $u$ with coefficients in $\mathbb{C}(z)$. Now we give a bound on the degree of the core series over $\mathbb{C}(z)$.
\begin{theorem}\label{5.16}
Let $M,N$ be as above, and set $h(u,z)=\det(1-zA)$, which can be viewed as a Laurent polynomial in $u$. Assume all small roots of $h(u,z)$ are $\rho_{1}(z),\ldots,\rho_{r}(z)$. Then
$$\sum_{n \geq 0}\ell(N\otimes M^n)z^n \in \mathbb{C}(z,\rho_{i}(z)|1 \leq i \leq r).$$
If $h(u,z)=u^{-\nu}h_1(u,z)$ such that $h_1(u,z)$ is a polynomial in $u$ which is not divisible by $u$, $\deg_u h_1(u,z)=D$, then
$$[\mathbb{C}(z,f(z)):\mathbb{C}(z)]\leq D!.$$
If moreover $A$ is a $1*1$ matrix, then 
$$[\mathbb{C}(z,f(z)):\mathbb{C}(z)]\leq D(D-1)\ldots (D-\nu+1).$$
Also, the above holds for the core series of $M$.
\end{theorem}
\begin{proof}
We notice that $|_{u=1},\frac{\partial}{\partial u}$ does not change the irreducible factor of the denominator of the expression, $\cdot g(u)$ can only introduce a factor $1-u$ on the denominator, and $[u^0]$ just evaluates the residue at small roots of the denominator. So when we start with $f(z)$, the only possible poles are $0$, small roots of $h(\xi^j_Qu,z)$; if all small roots of $h(u,z)$ are $\rho_{1}(z),\ldots,\rho_{r}(z)$, then all small roots of $h(\xi^j_Qu,z)$ are $\xi^{-j}_Q\rho_{1}(z),\ldots,\xi^{-j}_Q\rho_{r}(z)$. Evaluating residue at $u=\xi_Q^j\rho_i(z)$ results in a series in $\mathbb{C}(z,\xi_Q^j\rho_i(z))$. So the resulting series lies in $\mathbb{C}(z,\rho_{i}(z)|1 \leq i \leq r)$ since $\xi_Q \in \mathbb{C}$. We see $\{\rho_{i}(z)\}$ is a subset of all roots of a polynomial of degree $D$, so their extension degree is at most $D!$, so $[\mathbb{C}(z,f(z)):\mathbb{C}(z)]\leq D!$. If the size of $A$ is $1*1$, $h(u,z)=u^{-\nu}h_1(u,z)$ such that $h_1(u,z)$ is a polynomial in $u$ which is not divisible by $u$, then $h_1(u)=u^\nu-z-zh_2(u)$ where $h_2(u) \in u\mathbb{C}[u]$. For such polynomial, it has $\nu$ small roots coming from one polynomial of degree $D$, so the degree of extension is at most $D(D-1)\ldots (D-\nu+1)$. For core series, we can take $M=N$ and note that
$$c_M(z)=1+z\sum_{n \geq 0}\ell(N\otimes M^n)z^n,$$
so the core series also satisfies the same degree bound.
\end{proof}
\begin{remark}
The result that $\Omega^+$-algebraic modules have rational series can be explained in the following way: in the settings of the above theorem $h(u,z)=\det(1-zA) \in 1+(u,z)\mathbb{C}[u,z]$, so $h(u,z)$ itself produces no poles. The only possible pole appearing will be $u=0$, so its residue at elements in $\mathbb{C}(u,z)$ will always be rational.   
\end{remark}
One might ask if the series is algebraic over smaller base field. Since core series have rational coefficients, this is true by the following lemma.
\begin{lemma}
If a series $f(z) \in \mathbb{C}[[z]]$ is algebraic over $\mathbb{C}(z)$ and lies in $\mathbb{Q}[[z]]$, then it is algebraic over $\mathbb{Q}(z)$ with same degree. In particular, this holds for core series.       
\end{lemma}
\begin{proof}
We see $f(z)$ is algebraic if and only if there is a nonzero polynomial $P(x,y) \in \mathbb{C}[x,y]$ such that $P(z,f(z))=0 \in \mathbb{C}[[z]]$. By finiteness of coefficients we may assume all coefficients lie in a finitely generated field over $\mathbb{Q}$, say $L$. We may assume $L=\mathbb{Q}(x_1,\ldots,x_r,\alpha)$ where $x_1,\ldots,x_r$ is transcendental and $\alpha$ is algebraic over $x_1,\ldots,x_r$. After clearing denominators, we may assume the coefficients lie in $R_0=\mathbb{Q}[x_1,\ldots,x_r,\alpha]$. So $R_0$ is a hypersurface ring. Map $R_0$ to $R_0/\mathfrak{m}_0=K$ for some maximal ideal $\mathfrak{m}_0$ such that the image $\bar{P}(x,y)$ of $P(x,y)$ is not zero, then $K$ is a subfield of $\overline{\mathbb{Q}}$ and $\bar{P}(z,f(z))=0$. This says $f(z)$ is algebraic over $\overline{\mathbb{Q}}(z)$, hence algebraic over $\mathbb{Q}(z)$.

Let $z^jf^i(z)=\sum_{m \geq 0}a_{ijm}z^m$. Let $\phi: \mathbb{N}^2 \to \mathbb{N}$ be any bijection between these two countable sets, and $A=(a_{\phi^{-1}(n)m})_{m,n \in \mathbb{N}}$ be an infinite matrix over $\mathbb{Q}$. We see the degrees of $f(z)$ over $\mathbb{C}(z)$ and $\mathbb{Q}(z)$ are both given by some vanishing and nonvanishing condition on certain determinants of $A$, so they agree.
\end{proof}

\subsection{Computations}
In this subsection, we will compute some concrete examples of core series of $\Omega$-algebraic modules.

\begin{example}
Let $G=V_4$ be the Klein 4-group, $M=\Omega^2(k)\oplus (\Omega^{-1}(k))^2$. We see
$$\ell(\Omega^n(k))=2|n|+1$$
has period $1$, and $M$ lies in a single orbit $\{\Omega^n(k)\}_{n \in \mathbb{Z}}$. Set
$$f(u,z)=\frac{1}{1-z(u^2+2u^{-1})}=\sum_{n,m}c_{nm}u^mz^n.$$
We have
$$\sum_n(\sum_{m}ma_{m,n})z^n=\frac{\partial f}{\partial u}(1,z),$$
$$\sum_n(\sum_{m}[m]^-a_{m,n})z^n=[u^0]f(u,z)u/(1-u)^2=[u^{-1}]f(u,z)/(1-u)^2.$$
So the core series
\begin{align}\label{equation 5.2}
\sum_{n \geq 0}\ell(\operatorname{core}(M^{\otimes n}))z^n=\sum_{n,m}(2|m|+1)c_{nm}z^n=\sum_{n,m}(4[m]^-+2m+1)c_{nm}z^n
\end{align}
\begin{align*}
=4[u^{-1}]f(u,z)/(1-u)^2+2\frac{\partial f}{\partial u}(1,z)+f(1,z).    
\end{align*}
We can write
$$f(u,z)=\frac{g(u,z)}{h(u,z)}=\frac{u}{u-2z-zu^3}.$$
Note that $u-2z-zu^3$ has only one small root $u=\rho(z)$ satisfying $\rho(z)=2z+O(z^2)$, and it has degree $3$ since $u-2z-zu^3$ is irreducible in $\mathbb{C}(u,z)$. It is simple since $u-2z-zu^3$ and $\frac{\partial}{\partial u}(u-2z-zu^3)=1-3zu^2$ are coprime. Here $u=1$ is not a small root. So
\begin{align*}
[u^{-1}]f(u,z)/(1-u)^2=\operatorname{Res}_{u=\rho(z)}f(u,z)/(1-u)^2\\
=\frac{g(u,z)}{h'_u(u,z)(1-u)^2}|_{u=\rho(z)}=\frac{\rho}{(1-3z\rho^2)(1-\rho)^2}.
\end{align*}
The computation of the rest terms of \Cref{equation 5.2} is straightforward, so we get
\begin{align}\label{equation 5.3}
\sum_{n \geq 0}\ell(\operatorname{core}(M^{\otimes n}))z^n=\frac{4\rho}{(1-3z\rho^2)(1-\rho)^2}+\frac{1}{1-3z}
\end{align}
\begin{align*}
=\frac{4\rho}{(1+12z)-(8+6z)\rho+(4-3z)\rho^2}+\frac{1}{1-3z}.     
\end{align*}
Here the first term in the last line of \Cref{equation 5.3} cannot lie in $\mathbb{C}(z)$, otherwise $\rho$ has degree $2$ which is impossible. Therefore, the degree of the core series is $3$.
\end{example}
\begin{example}
Let $G=V_4$ be the Klein 4-group, $M=(\Omega^3(k))^2\oplus (\Omega^{-2}(k))^3$. The above computation is the same if we take $$f(u,z)=\frac{1}{1-z(2u^3+3u^{-2})}=\frac{u^2}{u^2-3z-2zu^5}.$$
In this case $f(u,z)$ has two simple poles since $u^2-3z-2zu^5$ has two small roots $u \sim \pm\sqrt{3}z^{1/2}$. Denote these roots by $\rho_\pm(z)$. Then
$$\operatorname{Res}_{u=\rho_\pm(z)}f(u,z)/(1-u)^{2}=\frac{\rho_\pm(z)}{(2-10z\rho_\pm(z)^3)(1-\rho_\pm(z))^2}.$$
So the final result is
$$\frac{2\rho_+}{(1-5z\rho_+^3)(1-\rho_+)^2}+\frac{2\rho_-}{(1-5z\rho_-^3)(1-\rho_-)^2}+\frac{1}{1-5z}.$$
\end{example}
\begin{example}
In \Cref{5.3}, we have got the core series of $M=\Omega(k)\oplus \Omega^{-1}(k)$, which is    
$$c_M(z)=(1+2z)(2/(1-4z^2)+4/(1-4z^2)^{3/2}).$$
Let $N=\operatorname{core} M^{\otimes 3}=\Omega^{-3}(k)\oplus (\Omega^{-1}(k))^3\oplus(\Omega(k))^3\oplus \Omega^{3}(k)$. Then letting $\omega=e^{2\pi \sqrt{-1}/3}$, we have
\begin{align*}
c_N(z)=\frac{1}{3}(c_M(z^{1/3})+c_M(\omega z^{1/3})+c_M(\omega^2z^{1/3}))\\
=\frac{2+16z}{1-64z^2} +\frac{4}{3}( \frac{1+2z^{1/3}}{(1-4z^{2/3})^{3/2}} +\frac{1+2\omega z^{1/3}}{(1-4\omega^2 z^{2/3})^{3/2}} +\frac{1+2\omega^2 z^{1/3}}{(1-4\omega z^{2/3})^{3/2}}).
\end{align*}
Set $x=z^{1/3}$. There are extension of fields as below:

\begin{center}
\begin{tikzcd}
\mathbb{C}(x,\sqrt{1-4x^2},\sqrt{1-4\omega x^2},\sqrt{1-4\omega^2x^2}) \arrow[d, no head] \arrow[dr, no head] & \\
\mathbb{C}(x^3,\sqrt{1-4x^2},\sqrt{1-4\omega x^2},\sqrt{1-4\omega^2x^2}) \arrow[d, no head] & \mathbb{C}(x) \arrow[dl, no head] \\
\mathbb{C}(x^3) &
\end{tikzcd}
\end{center}

The Galois group of $\mathbb{C}(x,\sqrt{1-4x^2},\sqrt{1-4\omega x^2},\sqrt{1-4\omega^2x^2})/\mathbb{C}(x^3)$ is $(\mathbb{Z}/2\mathbb{Z})^3\rtimes \mathbb{Z}/3\mathbb{Z}$,
where $\mathbb{Z}/3\mathbb{Z}$ acts on $(\mathbb{Z}/2\mathbb{Z})^3$ via permutation of lower indices. Here each $\mathbb{Z}/2\mathbb{Z}$ changes a sign of $\sqrt{1-\omega^i4x^2}$ and $\mathbb{Z}/3\mathbb{Z}$ permutes these $3$ square roots. We see
$$c_N(z)=c_N(x^3)$$
is not stable under any $\mathbb{Z}/2\mathbb{Z}$-action, so the Galois group of 
$$\mathbb{C}(x,\sqrt{1-4x^2},\sqrt{1-4\omega x^2},\sqrt{1-4\omega^2x^2})/\mathbb{C}(x^3,c_N(x^3))$$
is equal to $\mathbb{Z}/3\mathbb{Z}$. Therefore,
$$[\mathbb{C}(x^3,c_N(x^3)):\mathbb{C}(x^3)]=8.$$
Therefore $\deg_{\mathbb{C}(z)}c_N(z)=8$.
\end{example}
\begin{example}\label[example]{5.22}
Now we compute the core series in \cite[Example 15.1]{BensonGamma} and \cite[Section 5, Example (4)]{CHU}. Let $G=\langle g,h\rangle=\mathbb{Z}/3\times \mathbb{Z}/3,k=\mathbb{F}_3$, $R=kG=k[t_1,t_2]/(t_1^3,t_2^3)$, $M$ be given by the following diagram, where each point represents a $k$-basis element of $M$; multiplication by $t_1$ sends a basis element to the one at its lower left, and multiplication by $t_2$ sends a basis element to the one at its lower right.

\begin{center}
\begin{tikzpicture}[thick, line cap=round, shorten >=4pt, shorten <=4pt]

\coordinate (A) at (0,0);
\coordinate (B) at (1,1);
\coordinate (C) at (2,0);
\coordinate (D) at (-1,1);
\coordinate (E) at (-2,0);
\coordinate (F) at (-1,-1);

\draw (A) -- (B);
\draw (A) -- (D);
\draw (A) -- (F);
\draw (B) -- (C);
\draw (D) -- (E);
\draw (E) -- (F);

\fill (A) circle (2pt);
\fill (B) circle (2pt);
\fill (C) circle (2pt);
\fill (D) circle (2pt);
\fill (E) circle (2pt);
\fill (F) circle (2pt);

\node at (D) [below=6pt] {$a$};
\node at (B) [below=6pt] {$b$};

\end{tikzpicture}    
\end{center}

In other words, 
$$M=Ra\oplus Rb/(t_1^2a,t_2^2a,t_2a-t_1b,t_2^2b).$$
Denote $c=(t_1t_2a)^*,d=(t_2b)^*$. $M^*$ is given by the reflection of the diagram which is
\begin{center}
\begin{tikzpicture}[thick, line cap=round, shorten >=4pt, shorten <=4pt]

\coordinate (A) at (0,0);
\coordinate (B) at (-1,-1);
\coordinate (C) at (-2,0);
\coordinate (D) at (1,-1);
\coordinate (E) at (2,0);
\coordinate (F) at (1,1);

\draw (A) -- (B);
\draw (A) -- (D);
\draw (A) -- (F);
\draw (B) -- (C);
\draw (D) -- (E);
\draw (E) -- (F);

\fill (A) circle (2pt);
\fill (B) circle (2pt);
\fill (C) circle (2pt);
\fill (D) circle (2pt);
\fill (E) circle (2pt);
\fill (F) circle (2pt);

\node at (C) [above=6pt] {$d$};
\node at (F) [below=6pt] {$c$};

\end{tikzpicture}    
\end{center}
In other words, 
$$M^*=Rc\oplus Rd/(t_2^2c,t_1^2c-t_2d,t_2^2d,t_1d).$$
We can take the pullbacks of $M,M^*$ along $S=k[t_1,t_2]/(t_1^3) \to R=k[t_1,t_2]/(t_1^3,t_2^3)$. Viewing them as $S$-modules, we have
$$M=Sa\oplus Sb/(t_1^2a,t_2^2a,t_2a-t_1b,t_2^2b),$$
$$M^*=Sc\oplus Sd/(t_2^2c,t_1^2c-t_2d,t_2^2d,t_1d).$$

Let $N=k\uparrow_{\langle g \rangle}^G$. Working in the stable category, we see the $\Omega$-orbit of the following $3$ modules is closed under taking $\otimes M$:
$$M,M^*,N$$
and by \cite{BensonGamma} the matrix for the action of $\otimes M$ is
$$A=\begin{pmatrix}
u & u^{-1} & 1\\
u^{-1} & u & 1\\
0 & 0 & 3u
\end{pmatrix}.$$
We have
$$\ell(\Omega^n_R(N))=\begin{cases}
3 & n\equiv 0 \operatorname{mod}2\\
6 & n\equiv 1 \operatorname{mod}2.
\end{cases}$$
Now we compute
$$\ell(\Omega^n(M))=\ell(\Omega^{-n}(M^*)),\ell(\Omega^n(M^*))=\ell(\Omega^{-n}(M))$$
for $n \in \mathbb{N}$.

We notice that $t_2^2M=0$. Therefore, letting $S=k[t_1,t_2]/(t_1^3)$, we have $R=S/(t_2^3)$ where $t_2^3\in (t_1,t_2)\operatorname{ann}_{S}M$. Therefore, by \cite[Proposition 3.3.5]{elias1998six}, we have an equation of Poincare series
$$\beta^{R}_M(z)=\beta^{S}_M(z)/(1-z^2).$$
We use Macaulay 2 to find the resolution of $M$ over $S_1$, which is
$$\ldots S^4\xrightarrow[]{\begin{pmatrix}
-t_2^2 & t_1^2 & 0 & 0\\
t_1 & 0 & 0 & 0\\
0 & -t_2 & t_1 & -t_2^2\\
t_2 & 0 & 0 & t_1^2
\end{pmatrix}}S^4$$
$$\xrightarrow[]{\begin{pmatrix}
0 & t_1^2 & 0 & 0\\
t_1 & t_2^2 & 0 & 0\\
t_2 & 0 & t_1^2 & t_2^2\\
0 & -t_2 & 0 & t_1
\end{pmatrix}}S^4 \xrightarrow[]{\begin{pmatrix}
-t_2^2 & t_1^2 & 0 & 0\\
t_1 & 0 & 0 & 0\\
0 & -t_2 & t_1 & -t_2^2\\
t_2 & 0 & 0 & t_1^2
\end{pmatrix}}S^4 \xrightarrow[]{\begin{pmatrix}
0 & -t_1^2 & 0 & 0\\
t_1 & t_2^2 & 0 & 0\\
t_2 & 0 & t_1^2 & t_2^2\\
0 & -t_2 & 0 & t_1
\end{pmatrix}}S^4 $$
$$\xrightarrow[]{\begin{pmatrix}
0 & -t_2 & t_1 & -t_2^2\\
-t_2 & 0 & 0 & t_1^2\\
t_2^2 & t_1^2 & 0 & 0\\
t_1 & 0 & 0 & 0
\end{pmatrix}}S^4  \xrightarrow[]{\begin{pmatrix}
t_1^2 & t_2^2 & t_2 & 0\\
0 & 0 & -t_1 & t_2^2
\end{pmatrix}}S^2 \xrightarrow[]{(a,b)} M \to 0$$
where it becomes periodic afterwards. 
Similarly $M^*$ is annihilated by $t_2^2$, and the minimal resolution of $M^*$ over $S$ is given by
$$\ldots S^4\xrightarrow[]{\begin{pmatrix}
0 & t_1^2 & 0 & t_2\\
t_1 & t_2^2 & -t_2 & 0\\
0 & 0 & t_1^2 & t_2^2\\
0 & 0 & 0 & t_1
\end{pmatrix}}S^4 \xrightarrow[]{\begin{pmatrix}
-t_2^2 & t_1^2 & t_2 & 0\\
t_1 & 0 & 0 & -t_2\\
0 & 0 & t_1 & -t_2^2\\
0 & 0 & 0 & t_1^2
\end{pmatrix}}S^4 \xrightarrow[]{\begin{pmatrix}
0 & t_1^2 & 0 & t_2\\
t_1 & t_2^2 & -t_2 & 0\\
0 & 0 & t_1^2 & t_2^2\\
0 & 0 & 0 & t_1
\end{pmatrix}}S^4 $$
$$\xrightarrow[]{\begin{pmatrix}
-t_2^2 & t_1^2 & t_2 & 0\\
0 & 0 & 0 & t_1^2\\
t_1 & 0 & 0 & -t_2\\
0 & 0 & t_1 & -t_2^2
\end{pmatrix}}S^4  \xrightarrow[]{\begin{pmatrix}
0 & t_2^2 & 0 & t_1^2\\
t_1 & 0 & t_2^2 & -t_2
\end{pmatrix}}S^2 \xrightarrow[]{(c,d)} M \to 0.$$
Therefore,
$$\beta^{S}_M(z)=\beta^{S}_{M^*}(z)=2+4z+4z^2+4z^3+\ldots,$$
$$\beta^{R}_M(z)=\beta^{R}_{M^*}(z)=2+4z+6z^2+8z^3+\ldots,$$
$$\gamma^{R}_M(z)=\gamma^{R}_{M^*}(z)=2+2z+4z^2+4z^3+\ldots.$$
In other words,
$$\gamma^R_n(M)=\gamma^R_n(M^*)=\begin{cases}
n+2 & n=2n_0,n_0 \geq 0\\
n+1 & n=2n_0+1,n_0 \geq 0.
\end{cases}$$
So
\begin{align*}
\ell(\Omega^n_R(M))=\ell(\Omega^n_R(M^*))=\begin{cases}
9n+6 & n=2n_0,n_0 \geq 0\\
9n+3 & n=2n_0+1,n_0\geq 0.
\end{cases}
\end{align*}
Equivalently, for $n \in \mathbb{Z}$,
$$\ell(\Omega^n(M))=\ell(\Omega^{-n}(M^*))=9|n|+\frac{9}{2}+\frac{3}{2}(-1)^n.$$
Set $\ell_1(n)=\ell(\Omega^n(M)),\ell_2(n)=\ell(\Omega^n(M^*))=\ell_1(n),\ell_3(n)=\ell(\Omega^n(N))$.
We compute
$$\det(1-zA)=u^{-2}(1-3uz)(u-z-zu^{2})(u+z-zu^{2}).$$
We have
$$(1-zA)^*=\begin{pmatrix}
1-4uz+3u^2z^2 & zu^{-1}-3z^2 & z+z^2(u^{-1}-u)\\
zu^{-1}-3z^2 & 1-4uz+3u^2z^2 & z+z^2(u^{-1}-u)\\
0 & 0 & 1-2uz+(u^2-u^{-2})z^2
\end{pmatrix}.$$
Let $v=(1,0,0)$, then
$$v(1-zA)^*v_1^T=1-4uz+3u^2z^2,$$
$$v^T(1-zA)^*v_2^T=zu^{-1}-3z^2,$$
$$v(1-zA)^*v_3^T=z+z^2(u^{-1}-u),$$
$$f_i(u,z)=\frac{1}{\det(1-zA)}v(1-zA)^*v_i^T=\sum_{i,m,n}a_{i,m,n}u^mz^n,i=1,2,3.$$
Then the core series is given by
$$c_M(z)=1+zc_1(z),c_1(z)=\sum_{n \geq 0}\ell(\operatorname{core}M^{\otimes (n+1)})z^n$$
and
$$c_1(z)=\sum_{1 \leq i \leq 3,m,n}\ell_i(m)a_{i,m,n}z^n.$$
Now we proceed with computation. We have

\[
f_1(u,z)=\frac{1-4uz+3u^2z^2}{\det(1-zA)}
=\frac{u^2(1-uz)}{(u-z-zu^2)(u+z-zu^2)},
\]
\[
f_2(u,z)=\frac{zu^{-1}-3z^2}{\det(1-zA)}
=\frac{uz}{(u-z-zu^2)(u+z-zu^2)},
\]
and
\[
f_3(u,z)=\frac{z+z^2(u^{-1}-u)}{\det(1-zA)}
=\frac{uz}{(1-3uz)(u-z-zu^2)}.
\]

Now set
\[
J(u,z):=\frac{1}{1-z(u+u^{-1})}=\frac{u}{u-z-zu^2}=f_1+f_2,
\]
\begin{align*}
c_1(z)=\sum_m \ell_1(m)[u^m]f_1+\sum_m \ell_2(m)[u^m]f_2+\sum_m \ell_3(m)[u^m]f_3\\
=
\sum_m \ell_1(m)[u^m]J+\sum_m \ell_3(m)[u^m]f_3.
\end{align*}
Instead of applying formal Cauchy's residue theorem, we will try to expand $J$ as series in $u,z$. Let
\[
s=\sqrt{1-4z^2},
\qquad
\alpha=\frac{1-s}{2z}=z+z^3+2z^5+\cdots .
\]
Then $\alpha,\alpha^{-1}$ are two roots of $-zu^2+u-z=0$ and $\alpha^{-1}-\alpha=s/z$. So
\begin{align*}
J(u,z)=\frac{u}{-z(u-\alpha)(u-\alpha^{-1})}=\frac{u(\alpha^{-1}-\alpha)}{-s(u-\alpha)(u-\alpha^{-1})}\\
=\frac{u}{s}(\frac{1}{u-\alpha}-\frac{1}{u-\alpha^{-1}})=\frac{1}{s}(\frac{\alpha u}{1-\alpha u}+\frac{1}{1-\alpha u^{-1}}).
\end{align*}
Since \(\alpha\in z\mathbb C[[z]]\), in \(\mathbb C((u))((z))\) we have
\[
J(u,z)=\frac1s\sum_{m\in\mathbb Z}\alpha^{|m|}u^m.
\]
Hence
\[
J(1,z)=\sum_m [u^m]J=\frac{1}{1-2z},J(-1,z)=\sum_m (-1)^m[u^m]J=\frac{1}{1+2z},
\]
and
\[
\sum_m |m|[u^m]J
=
\frac{2}{s}\sum_{m\ge1}m\alpha^m
=
\frac{2\alpha}{s(1-\alpha)^2}
=
\frac{2z}{(1-2z)\sqrt{1-4z^2}}.
\]
Therefore,
\begin{align*}
\sum_m \ell_1(m)[u^m]J=\frac{9J(1,z)}{2}+\frac{3J(-1,z)}{2}+9\sum_m |m|[u^m]J\\
=\frac{9}{2(1-2z)}+\frac{3}{2(1+2z)}+\frac{18z}{(1-2z)\sqrt{1-4z^2}}\\
=\frac{6(1+z)}{(1-4z^2)}+\frac{18z}{(1-2z)\sqrt{1-4z^2}}.
\end{align*}
Since
\[
\ell_3(m)=
\begin{cases}
3,& m\equiv 0\pmod 2\\
6,& m\equiv 1\pmod 2
\end{cases}=\frac92-\frac32(-1)^m,
\]

\[
\sum_m \ell_3(m)[u^m]f_3
=
\frac92 f_3(1,z)-\frac32 f_3(-1,z)=
\frac{3z(1+10z+6z^2)}{(1-4z^2)(1-9z^2)}.
\]

Combining all equations above, we get
\begin{align*}
c_1(z)
=
\left(
\frac{6(1+z)}{(1-4z^2)}+\frac{18z}{(1-2z)\sqrt{1-4z^2}}
\right)
+
\frac{3z(1+10z+6z^2)}{(1-4z^2)(1-9z^2)}\\
=\frac{3(2+3z-8z^2-12z^3)}{(1-4z^2)(1-9z^2)}
+\frac{18z}{(1-2z)\sqrt{1-4z^2}}
\end{align*}
and $c_M(z)=1+zc_1(z)$. We see $\deg_{\mathbb
{C}(z)}c_M(z)=2$.
\end{example}

\section{Iterated Syzygies}
In the previous section, we focus on $\Omega$-algebraic modules. In this section, we compute a core-series of an iterated syzygy, which is not $\Omega$-algebraic.

\textbf{Settings.} In this section, let $H=(\mathbb{Z}/p^{e_1})g_1\times (\mathbb{Z}/p^{e_2})g_2\times\ldots\times(\mathbb{Z}/p^{e_r})g_r$ for some $e_1,e_2,\ldots,e_r\geq 1,r\geq 2$, $R=kH=k[t_1,\ldots,t_r]/(t_1^{p^{e_1}},\ldots,t_r^{p^{e_r}})$ where $t_i=g_i-1$. Let $G=(\mathbb{Z}/p^{f_1})g_1\times (\mathbb{Z}/p^{f_2})g_2\times\ldots\times(\mathbb{Z}/p^{f_r})g_r$ for some $f_i>e_i$, $S=kG=k[t_1,\ldots,t_r]/(t_1^{p^{f_1}},\ldots,t_r^{p^{f_r}})$ where $t_i=g_i-1$. There is a natural surjection $G \to H$ sending $g_i$ to $g_i$ which makes $R$ a quotient ring of $S$. So, any $R$-module is naturally an $S$-module. On the level of representations this identification is just the inflation of representations, and we will omit the inflation symbol here. In this sense, the following symbol
$$\Omega^j_S\Omega^i_R(k)$$
makes sense as an $S$-module for any $i,j \in \mathbb{Z}$. Let $P=k[t_1,\ldots,t_r]$ be a polynomial ring in $t_i$'s, which surjects onto $S$ and onto $R$. 

We see indecomposability is preserved under inflation and taking syzygies. Therefore, $\Omega^j_S\Omega^i_R(k)$ is indecomposable for all $i,j$. Also, $\otimes_k$ commutes with inflations, so we do not distinguish between tensor products as $R$-modules and as $S$-modules.
\begin{theorem}\label{6.1}
The set of isomorphic classes of direct sums of modules in the following classes 
$$\Omega^j_S\Omega^i_R(k),\Omega^j_S(R),S$$
is closed under $\otimes_k$.
\end{theorem}
\begin{proof}
From \Cref{4.8} we see in the $R$-module category, 
$$\Omega^{i_1}_R(k)\otimes_k \Omega^{i_2}_R(k)=\Omega^{i_1+i_2}_R(k)\oplus R^s$$
for some $s \in \mathbb{N}$. Therefore, in the stable module category of $S$ we have
$$\Omega^{j_1}_S\Omega^{i_1}_R(k)\otimes_k \Omega^{j_2}_S\Omega^{i_2}_R(k)=\Omega^{j_1+j_2}_S\Omega^{i_1+i_2}_R(k)\oplus \Omega^{j_1+j_2}_S(R^s).$$
In other words, in the module category we have
$$\Omega^{j_1}_S\Omega^{i_1}_R(k)\otimes_k \Omega^{j_2}_S\Omega^{i_2}_R(k)=\Omega^{j_1+j_2}_S\Omega^{i_1+i_2}_R(k)\oplus \Omega^{j_1+j_2}_S(R^s)\oplus S^t$$
for some $s,t \in \mathbb{N}$. Similarly we have
$$\Omega^{j_1}_S\Omega^{i_1}_R(k)\otimes_k \Omega^{j_2}_S(R)=\Omega^{j_1+j_2}_S(R^{\ell(\Omega^{i_1}_R(k))})\oplus S^t$$
and
$$\Omega^{j_1}_S(R)\otimes_k \Omega^{j_2}_S(R)=\Omega^{j_1+j_2}_S(R^{\ell(R)})\oplus S^t$$
which concludes the proof.
\end{proof}
\begin{theorem}
The indecomposable modules listed in \Cref{6.1} are pairwise non-isomorphic.     
\end{theorem}
\begin{proof}
Assume
$$\Omega^{j_1}_S\Omega^{i_1}_R(k)\cong \Omega^{j_2}_S\Omega^{i_2}_R(k).$$
We claim this implies $j_1=j_2$ and $i_1=i_2$.
After applying $\Omega^{-j_2}_S$ on both sides, for $j=j_1-j_2$ we have
$$\Omega^{j}_S\Omega^{i_1}_R(k)\cong \Omega^{i_2}_R(k).$$
Now we take the tensor product with $\Omega^{-i_2}_R(k)$, then in stable category we have
$$\Omega^{j}_S\Omega^{i_1-i_2}_R(k)\oplus \Omega^{j}_S(R)^{s_1} \cong k\oplus R^{s_2}.$$
We see
$$\Omega^{j}_S\Omega^{i_1-i_2}_R(k)\cong k,\Omega^{j}_S(R) \cong R$$
since $\ell(R)$ does not divide the dimensions of the first two modules, but divides the dimensions of the last two modules. This implies $j=i_1-i_2=0$.

We have
$$\Omega^{j_1}_S\Omega^{i_1}_R(k)\not\cong \Omega^{j_2}_S(R)$$
since $\ell(R)$ divides the length of the right side, but does not divide the dimension of the left side. $\Omega^j_S\Omega^i_R(k),\Omega^j_S(R)$ is not projective over $S$ since any finitely generated $S$-module of finite projective dimension over $S$ is free over $S$, so is not an $R$-module.
\end{proof}
\begin{corollary}
If $i \neq 0$, then the $S$-module $\Omega^j_S\Omega^i_R(k)$ is not $\Omega$-algebraic.    
\end{corollary}

Since $\Omega^j_S\Omega^i_R(k)$ is not $\Omega$-algebraic, the direct summand of such modules may have transcendental core series. 

Now we compute the following example to show the core series can indeed be transcendental. Let $H=V_4=\mathbb{Z}/2\times \mathbb{Z}/2$ and $G=\mathbb{Z}/4\times \mathbb{Z}/4$. In this case, $\Omega^i_R(k)$ is a $P=k[t_1,t_2]$-module, which has 3 nonzero Betti numbers. Since it is torsion over $P$, its alternating sum of Betti numbers over $P$ is $0$.

$$\beta^P_0(\Omega^i_R(k))=\beta^R_0(\Omega^i_R(k)),\beta^P_2(\Omega^i_R(k))=\ell(\operatorname{soc}\Omega^i_R(k))=\beta^R_0(\Omega^{-i}_R(k)).$$
Therefore, we get
$$\beta^P_0(\Omega^i_R(k))=i+1,\beta^P_1(\Omega^i_R(k))=2i+1,\beta^P_2(\Omega^i_R(k))=i,i \geq 1$$
and
$$\beta^P_0(\Omega^{-i}_R(k))=i,\beta^P_1(\Omega^{-i}_R(k))=2i+1,\beta^P_2(\Omega^{-i}_R(k))=i+1,i \geq 1.$$
When $i=0$, we get
$$\beta^P_0(k)=1,\beta^P_1(k)=2,\beta^P_2(k)=1.$$
Now we compute some series as below. Note that $S$ is $P$ modulo a regular sequence which falls into $(t_1,t_2)\operatorname{ann}_P(\Omega^i_R(k))$. Therefore,
$$\sum_{j \geq 0}\beta^S_j(\Omega^i_R(k))z^j=\sum_{j \geq 0}\beta^S_j(\Omega^i_R(k))z^j\cdot \frac{1}{(1-z^2)^2}$$
and
$$\sum_{j \geq 0}\gamma^S_j(\Omega^i_R(k))z^j=\sum_{j \geq 0}\beta^S_j(\Omega^i_R(k))z^j\cdot \frac{1}{(1-z^2)^2(1+z)}$$
and
\begin{align*}
\ell(\Omega^j_S\Omega^i_R(k))z^j=\sum_{j \geq 0}(\ell(S)\gamma^S_{j-1}(\Omega^i_R(k))+(-1)^j\ell(\Omega^i_R(k)))z^j\\
=\frac{\ell(S)z\cdot\sum_{j \geq 0}\beta^S_j(\Omega^i_R(k))z^j}{(1-z^2)^2(1+z)}+\ell(\Omega^i_R(k))\cdot\frac{1}{1+z}.
\end{align*}
Set $$\ell(\Omega^j_S\Omega^i_R(k))=a_{ij},i,j \in \mathbb{Z}.$$
Then
$$\sum_{j \geq 0}a_{ij}z^j=\begin{cases}
\frac{16z(i+1+iz)}{(1-z^2)^2}+\frac{2|i|+1}{1+z}    & i\geq 1\\
\frac{16z(1+z)}{(1-z^2)^2}+\frac{1}{1+z}     & i=0\\
\frac{16z(|i|+(1+|i|)z)}{(1-z^2)^2}+\frac{2|i|+1}{1+z}     & i\leq -1.\\
\end{cases}$$
We also have
$$\Omega^j_S\Omega^i_R(k)^*=\Omega^{-j}_S\Omega^{-i}_R(k),a_{-i,-j}=a_{ij}.$$
This gives $\ell(\Omega^j_S\Omega^i_R(k))$ for $i,j \in \mathbb{Z}$. 

Then, from the fact that $\Omega^m_S(M)\otimes \Omega^n_S(N)=\Omega^{m+n}_S(M\otimes N)$ in the stable category, we see
$$P_1(\Omega_S)(M)\otimes P_2(\Omega_S)(N)=P_1P_2(\Omega_S)(M\otimes N)$$
for Laurent polynomials $P_1,P_2 \in \mathbb{N}[u,u^{-1}]$. In particular,
$$P(\Omega_S)(M)^{\otimes n}=P^n(\Omega_S)(M^{\otimes n}).$$
Let $N=\Omega^1_R(k)\oplus \Omega^{-1}_R(k)$. Then
$$N^{\otimes n}=\sum_{0 \leq \alpha \leq n}\Omega^{n-2\alpha}_R(k)^{\oplus {n\choose \alpha}}\oplus R^{s_n}$$
where 
$$s_n=\frac{1}{4}(6^n-c^H_n(N)).$$
Here
$$c^H_n(N)=2^n+2n{n\choose n/2}$$
for $n$ even and
$$c^H_n(N)=2^n+4n{n-1\choose (n-1)/2}$$
for $n$ odd.

Let $M=\Omega^1_S(N)\oplus \Omega^{-1}_S(N)$. Then in the stable module category over $S$ we have
\begin{align*}
M^{\otimes n}=\sum_{0 \leq \delta \leq n}\Omega^{n-2\delta}_S(\sum_{0 \leq \alpha \leq n}\Omega^{n-2\alpha}_R(k)^{\oplus {n\choose \alpha}}\oplus R^{s_n})^{\oplus {n\choose \delta}}\\
=\sum_{0 \leq \alpha,\delta \leq n}\Omega^{n-2\delta}_S\Omega^{n-2\alpha}_R(k)^{\oplus {n\choose \alpha}{n\choose \delta}}\oplus\sum_{0 \leq \delta \leq n}\Omega^{n-2\delta}_S(R)^{\oplus{n\choose \delta}s_n}.
\end{align*}
To compute the core series of $M$, we also need to compute the lengths of $\Omega^i_S(R)$. Instead of direct computation, we observe that there is an exact sequence
$$1 \to H \to G \to H \to 1,$$
and therefore as $S$-modules, $R \cong k\uparrow_{H}^G$, so
$$\Omega^i_S(R)\cong \Omega^i_R(k)\uparrow_{H}^G$$
and
$$\ell(\Omega^i_S(R))=[G:H]\ell(\Omega^i_R(k))=4(2|i|+1).$$
In sum, we have
$$c^G_n(M)=\sum_{0 \leq \alpha,\delta \leq n}{n\choose \delta}{n \choose \alpha}a_{n-2\alpha,n-2\delta}+(\sum_{0 \leq \delta \leq n}4{n\choose \delta}(2|n-2\delta|+1))\cdot s_n.$$
The rest part is just computation. Denote
$$A=\sum_{0 \leq \delta \leq n}{n \choose \delta}|n-2\delta|=\begin{cases}
n{n\choose n/2} & n \textup{ even}\\
(n+1){n\choose (n-1)/2} & n \textup{ odd,}
\end{cases}$$
$$B=\sum_{0 \leq \delta \leq n}{n\choose \delta}=2^n.$$
Note that
$$(\sum_{0 \leq \delta \leq n}4{n\choose \delta}(2|n-2\delta|+1))=4(2A+B)$$
$$c^H_n(N)=B+2A.$$
If $n$ is odd, so is $j=n-2\delta,i=n-2\alpha$. Therefore, we can verify
$$a_{ij}=\begin{cases}
8(|i|+1)(|j|+1)-(2|i|+1) & ij>0 \textup{ odd}\\
=8|ij|+8|j|+6|i|+7\\
8(|i|)(|j|+1)-(2|i|+1) & ij<0 \textup{ odd}\\
=8|ij|+6|i|-1
\end{cases}$$
or
$$a_{ij}=8|ij|+6|i|-1+\mathbf{1}_{ij>0}(8|j|+8).$$
So
$$\sum_{0 \leq \alpha,\delta \leq n}{n\choose \delta}{n \choose \alpha}a_{n-2\alpha,n-2\delta}=8A^2+6AB-B^2+\frac{1}{2}(8AB+8B^2)=8A^2+10AB+3B^2.$$
So
\begin{align*}
c^G_n(M)=8A^2+10AB+3B^2+4(2A+B)\cdot \frac{1}{4}(6^n-B-2A)\\
=4A^2+6AB+2B^2+6^n(2A+B).
\end{align*}
If $n$ is even, so is $j=n-2\delta,i=n-2\alpha$. Set
$$C={n\choose n/2}.$$
When $i,j$ are both even, we can verify
$$a_{ij}=\begin{cases}
8|i||j|+2|i|+1 & i>0,j\geq 0 \textup{ or }i<0,j\leq 0\\
8|i||j|+8|j|+2|i|+1 & i\leq 0,j \geq 0 \textup{ or } i \geq 0,j \leq 0
\end{cases}$$
or
$$a_{ij}=8|i||j|+8|j|+2|i|+1-\mathbf{1}_{(ij> 0)  }\cdot8|j|.$$
We have
\begin{align*}
\sum_{0 \leq \alpha,\delta \leq n, (n-2\delta)(n-2\alpha)> 0}{n\choose \delta}{n \choose \alpha}|n-2\delta|\\
=\frac{1}{2}\sum_{0 \leq \alpha,\delta \leq n, (n-2\delta)(n-2\alpha)\neq 0}{n\choose \delta}{n \choose \alpha}|n-2\delta|\\
=\frac{1}{2}\sum_{0 \leq \alpha,\delta \leq n, (n-2\alpha)\neq 0}{n\choose \delta}{n \choose \alpha}|n-2\delta|\\
=(B-C)\cdot\frac{1}{2}\sum_{0 \leq \delta \leq n}{n\choose \delta}|n-2\delta|\\
=\frac{1}{2}A(B-C).
\end{align*}
So
$$\sum_{0 \leq \alpha,\delta \leq n}{n\choose \delta}{n \choose \alpha}a_{n-2\alpha,n-2\delta}=8A^2+10AB+B^2-4A(B-C)=8A^2+6AB+B^2+4AC,$$
\begin{align*}
c^G_n(M)=8A^2+6AB+B^2+4AC+4(2A+B)\cdot \frac{1}{4}(6^n-B-2A)\\
=4A^2+2AB+4AC+6^n(2A+B).
\end{align*}
Here we see $A,C$ are two combinatorial numbers and $B$ is a power, so the generating function of $AB,B^2,6^nA,6^nB$ are all algebraic. However, the generating function of $A^2$ is transcendental. Actually the generating function of $A^2+AC$ is also transcendental since $C=\frac{1}{n}A$ when $n$ is even, and they do not cancel asymptotically.

One way of proving this fact is to check the asymptotic behavior of $4A^2$. The Stirling series says
$$n!\sim \sqrt{2\pi n}(\frac{n}{e})^n(1+\frac{1}{12n}+\frac{1}{288n^2}-\frac{139}{51840n^3}+O(n^{-4}))$$
which gives
$${2n\choose n}^2=\frac{16^n}{\pi n}(1-\frac{1}{4n}+\frac{1}{32n^2}+O(n^{-3})),$$
so for some $c_1,c_2,c_3,\alpha$ we will have
$$4A^2-c_1n\alpha^n-c_2\alpha^n=c_3\alpha^nn^{-1}+O(\alpha^nn^{-2}).$$ 
However, if a series $f(z)=\sum_{n \geq 0}a_nz^n$ is algebraic and has a singularity at $z=1/\alpha$, then near $1/\alpha$ it can be expanded as a Puiseux series
$$f(z)=\sum_{j \geq m}(1-\alpha z)^{j/d}$$
which gives
$$a_n\sim Cn^{\beta}\alpha^n,\beta \notin \mathbb{Z}_{<-1},$$
and if 
$$a_n\sim Cn^{-1}\alpha^n$$
this would give us $f(z) \sim -C\log(1-\alpha z)$, which is not algebraic. Relevant discussions can be seen in \cite{TRANS}.

Therefore, the generating function of $4A^2-c_1n\alpha^n-c_2\alpha^n$ is transcendental over $\mathbb{C}(z)$. And the generating function of $c_1n\alpha^n+c_2\alpha^n$ is rational over $\mathbb{C}(z)$. Therefore, the generating function of $4A^2$ is transcendental. We can prove the transcendence of the generating function of $4A^2+4AC$ in the same way.
\begin{theorem}\label{6.4}
Let
$$M=\Omega_S\Omega_R(k)\oplus \Omega_S\Omega^{-1}_R(k)\oplus\Omega_S^{-1}\Omega_R(k)\oplus \Omega_S^{-1}\Omega^{-1}_R(k).$$
Then the core series of $M$ is transcendental.
\end{theorem}
\begin{remark}
The transcendence here comes from the fact that the Hadamard product of two algebraic series is not always algebraic-the generating function of $A$ is algebraic while the generating function of $A^2$ is not. One set of series which contains algebraic series and is closed under Hadamard product is the $D$-finite series, defined as below:

A series $f(z)$ is $D$-finite if it satisfies a differenial equation
$$P_r(z)f^{(r)}(z)+\ldots+P_1(z)f^{(1)}(z)+P_0(z)f(z)=0.$$
As a particular case, for
$$M=\Omega_S\Omega_R(k)\oplus \Omega_S\Omega^{-1}_R(k)\oplus\Omega_S^{-1}\Omega_R(k)\oplus \Omega_S^{-1}\Omega^{-1}_R(k),$$
its core series is $D$-finite. We may ask the following questions.
\end{remark}
\begin{question}
Let $M$ be a direct sum of iterated syzygies $\Omega_S^j\Omega^i_R(k)$ for any finite abelian $p$-group surjections $G \twoheadrightarrow H$. Is the core series of $M$ always $D$-finite?  
\end{question}
We may further consider syzygies iterated multiple times, that is, modules of the form
$$\Omega^{i_s}_{R_s}\Omega^{i_{s-1}}_{R_{s-1}}\ldots \Omega^{i_1}_{R_1}(k)$$
for a chain of surjection of groups
$$G_s \twoheadrightarrow G_{s-1}\twoheadrightarrow\ldots\twoheadrightarrow G_1.$$
\begin{question}
Is there a core series which is not $D$-finite?    
\end{question}
\section*{Acknowledgement}
The author would express his sincere thanks to Dave Benson for his valuable suggestions, which inspired the the calculations in Chapter 5. The author would like to thank Srikanth Iyengar and Josh Pollitz for helpful discussions.
\bibliographystyle{plain}
\bibliography{reference}
\end{sloppypar}
\end{document}